\documentclass[10pt]{amsart}
\usepackage{amsfonts}
\usepackage{amsmath}
\usepackage{amssymb}
\usepackage[utf8]{inputenc}
\usepackage[english]{babel}
\usepackage{enumerate}
\usepackage{mathtools}
\usepackage{microtype}
\usepackage[
pagebackref,
bookmarks=true,         %generates bookmarks for all entries in the "Table of Contents"
bookmarksnumbered=true, %includes chapter-/header-/section-/subsection-/... -
colorlinks=true,
pdfstartview=FitV,
linkcolor=blue, citecolor=blue, urlcolor=blue]{hyperref}

\usepackage[msc-links,backrefs]{amsrefs}

\newtheorem{theorem}{Theorem}[section]

\newtheorem{corollary}[theorem]{Corollary}

\newtheorem{lemma}[theorem]{Lemma}

\newtheorem{problem}[theorem]{Problem}
\newtheorem{proposition}[theorem]{Proposition}

\theoremstyle{definition}
\newtheorem{definition}[theorem]{Definition}

\theoremstyle{remark}
\newtheorem{remark}[theorem]{Remark}

\def\sgn{\mathrm{sgn\, }}
\def\RR{\mathbb{R}}

\def\ZZ{\mathbb{Z}}
\def\EE{\mathbb{E}}

\def\wt{\widetilde}

\def\cB{{\mathcal B}}

\def\cF  {{\mathcal F}}

\def\si{{\sigma}}

\def\cL{{\mathcal L}}

\def\al{{\alpha}}

\def\si{{\sigma}}

\def \eref#1{\hbox{(\ref{#1})}}

\def \eref#1{\hbox{(\ref{#1})}}

\def\si{{\sigma}}
\def\al{{\alpha}}

\def\wt{\widetilde}

\def\cB{{\mathcal B}}

\def\cF  {{\mathcal F}}

\def\si{{\sigma}}

\def\cL{{\mathcal L}}

\def\al{{\alpha}}

\def\si{{\sigma}}

\def \eref#1{\hbox{(\ref{#1})}}

\def\om{{\omega}}

\def\XXint#1#2#3{{\setbox0=\hbox{$#1{#2#3}{\int}$ }
\vcenter{\hbox{$#2#3$ }}\kern-.6\wd0}}

\begin{document}

\title{Stochastic differential equation for Brox diffusion}
\author[Y. Hu]{Yaozhong Hu}
\thanks{Y. Hu is partially supported by a grant from the Simons
Foundation \#209206 and a General Research Fund of University of
Kansas.}
\author[K. L\^e]{Khoa L\^e}
\address{Department of Mathematics, The University of Kansas, Lawrence, Kansas, 66045, USA}
\email{yhu@ku.edu, khoale@ku.edu}
\author[L. Mytnik]{Leonid Mytnik}
\thanks{L. Mytnik is partially supported by a grant from the Israel Science Foundation.}
\address{ Industrial Engineering and Management, Technion - Israel Institute of Technology
Technion City, Haifa 32000, Israel}
\email{leonid@ie.technion.ac.il}

\date{}%\today}

\begin{abstract}
This paper studies the  weak and  strong solutions   to the
stochastic differential equation $ dX(t)=-\frac12 \dot
W(X(t))dt+d\mathcal{B}(t)$,  where $(\mathcal{B}(t), t\ge 0)$ is a standard Brownian
motion and $W(x)$ is a two sided Brownian motion,  independent of
$\mathcal{B}$.  It is shown  that the It\^o-McKean representation associated with  any Brownian
motion (independent of $W$) is a weak solution to the
above equation. It is also shown that there exists a unique strong solution to the equation. It\^o calculus for the solution is developed. For dealing with the singularity of drift term
$\int_0^T \dot W(X(t))dt$, the main idea is to use the concept of local time
together with the  polygonal approximation
  $W_\pi$.  Some new results   on the local time of Brownian
motion needed in our proof are  established.
\end{abstract}
\maketitle
\setcounter{tocdepth}{1}
\tableofcontents

\section{Introduction}
Ever since the work of Sinai \cite{Sinai}
on the random walk in
random medium
there has been a great amount of work on random
processes  in a random environment.   One of the continuous time and continuous space
analogues of Sinai's random walk  is the Brownian motion in a white noise
medium, namely, the Brox diffusion,  which   can be described briefly
as follows.  Let $(\mathcal{B}(t), t\ge 0)$ be a  one dimensional  standard
Brownian motion and let $(W(x)\,, x\in \RR)$ be a two sided  one
dimensional  Brownian motion, independent of
$\mathcal{B}$.   Its   derivative   $\dot W(x)$  with respect to $x$ in the sense of Schwartz distribution is called  the white noise (see \cite{HKPS}).
The Brox diffusion is a diffusion process $X(t)$
determined formally by the following stochastic differential
equation
\begin{equation}
X(t)=-\frac12 \int_0^t  \dot W(X(s))ds+\mathcal{B}(t)\,. \label{equation}
\end{equation}
Throughout the paper, we assume the initial condition
 $X(0)=0$ for simplicity.  Since $\dot W$ is a distribution  (generalized function),
the conventional theory of stochastic differential equations does
not apply to the above equation  \eref{equation}.

In the case $W$ is nice (for example, $\dot W(x)$ is  deterministic and
globally Lipschitz continuous),  then the solution $X(t)$  to \eref{equation} exists
uniquely and it  is a Markov process
% In \cite{Bro} a Markov  process   has been defined
 % \footnote{LM Where it was defined this way? We should
%give a reference!}
   with generator
\begin{equation}
A =\frac12 e^{W(x)}\frac{d}{dx} \left(e^{-W(x)} \frac{d}{dx}\right)\,.
\end{equation}
%\footnote{LM Here is the 1st time where we confuse two Brownian
%motions! $B$ below is not the $B$ above!  So, I would say here
%something like
% "Let $(\Omega, {\mathcal F},P)$  be a probability space. First we will give construction of $X$ via time
%  change on this probability  space.  Define $W=(W(x)\,, x\in \RR)$ and
%   $\tilde B=(\tilde B(t)\,, t\geq 0)$ on  $(\Omega, {\mathcal F},P)$  where $W$ is as above and $\tilde B$ is a standard
%one-dimensional Brownian motion independent of $W$". Probably we
%also need to put filtration here for $\tilde B$.   Then I would
%change $B$ below for $\tilde B$. The other option would be to change
%$B$ for $\tilde B$ in (1.1) but I guess it is less natural.}
%\footnote{LM I would reformulate it as below. See if you agree}
In \cite{Bro}, the process $X(t)$ defined (formally) by \eref{equation} is identified as a Feller
diffusion  with the above generator $A$.  The It\^o-McKean's construction of this Feller diffusion
from a Brownian motion via scale-transformation and time change is particularly used there.
%
%
%From the general theory of Markov process (see for example  \cite{Fuk})
%even when $W$ is just continuous function, the above generator defines a Markov process.
%% \footnote{LM slight change}
Let us    briefly recall this construction.  % of such a process, which is usually called the It\^o-McKean
%representation.
%~\footnote{LM2 do we have exact reference to this construction in Ito-McKean book?}
Let $B$ be a Brownian motion defined on a probability space $(\Omega, \mathcal{F},\mathcal{F}_t, P)$,
independent of $(W(x), x\in\RR)$
%\footnote{LM4 New sentence}
(Note that, if it is not stated otherwise,
we assume throughout the paper that
   $(W(x), x\in\RR)$ is a two sided Brownian motion).  We define the spatial transformation
\begin{equation}
S_W(x)=\int_0^x e^{W(z) } dz,\ \label{sw}
\end{equation}
and the time change
\begin{equation}
T_{W, B}(t)=\int_0^t e^{-2W\circ S_W^{-1}(B(s) )} ds\,.\label{twb}
\end{equation}
%then the Markov process $X(t)$ has the following representation.
Then, the Feller diffusion $(X(t)\,, t\ge 0)$ associated with  \eref{equation} is  represented as
\begin{equation}
X(t)=S_W^{-1} \circ B\circ T_{W, B}^{-1}(t)\,, \quad 0\le
t<\infty\,. \label{sol-rep}
\end{equation}
%It is well known   that  $X$ defined by~(\ref{sol-rep}) is a Markov process with generator $A$
%for a particular continuous $W$ (see e.g. \cite{Bro})
We shall call \eref{sol-rep} the It\^o-McKean representation of the Feller diffusion. With this representation Th. Brox (in \cite{Bro}) studied the limit of the  scaled  process $\al^{-2} X(e^\al)$ (and the limit of the form $\al^{-2} X(e^{\al h(\al)} )$,  where $h(\al)\rightarrow 1$)    as $\al\rightarrow \infty$.

After this  work of Brox (\cite{Bro}) there have been a number of papers
devoted to the study of the process $X(t)$ defined by \eref{sol-rep}.
Let us only mention the papers \cites{AD,Di,Shi} where the local time of
$X(t)$ is studied. Some  ideas in these papers  will be used later. Let us also mention that about the same time as \cite{Bro} the  process $X(t)$ was also  studied in the paper
\cite{Sch}.

It may be    interesting to note that  if $W$ were continuously differentiable, it could be easily checked by It\^o's calculus that such an $X$ defined by \eref{sol-rep} is a weak solution to \eref{equation}  (see Remark~\ref{rem.32} (i) in Section~\ref{sec:prelim}).

By definition a diffusion is a Markov process with continuous sample paths. Probabilists are
interested in more detailed properties of the sample paths.
By fixing an almost sure realization of two-sided Brownian motion $W$,
the equation \eref{equation} can be considered as a
stochastic differential equation with singular
drift in the form
 \begin{equation}\label{eqn.russo}
        X_t=X_0+\int_0^t \sigma(X_s)d\mathcal{B}(s)+\int_0^t b'(X_s)ds,
    \end{equation}
    where $\mathcal{B}$ is a Brownian motion,  and $\sigma$ and $b$ are continuous function.
In fact,  there
have been already a number of work on such (one dimensional)
 equations (see e.g. \cite{basschen},
\cite{rf1}, \cite{rf2}, \cite{russo-trutnau},
and the references
therein).  In some cases strong existence and  uniqueness has been proved for such equations.
In  the case   $\sigma\equiv 1$ (which, in fact, is the situation
in~\eref{equation})  if   $b$ is H\"older continuous of order $\alpha$ for some $\alpha>1/2$,
then  the existence and  uniqueness of the strong solution   to~\eref{eqn.russo}  were
derived   in   \cite{basschen}. Under similar conditions, these results have been
also proved in~\cite{russo-trutnau}. 
%\footnote{HL: this senctence is strange, we tried to rephrase it but couldn't settle to a final version. We decided to keep it as is. }
However, it seems that in the case of the function $b$ being
less regular than H\"older of order $1/2$,
the  representation for $X$ which is known is via solution of certain martingale problem,  or
time change analogous to \eref{sol-rep} or via weak solution to
\eref{eqn.russo}, where the last term on the right hand side of the equation is defined as an
extension of a certain map (see e.g. Corollary~3.4 and Remark~3.5 in~\cite{rf1} or
Corollary~5.13 and Remark~5.14 in~\cite{russo-trutnau}.)
We would like to mention
that existence and uniqueness of the
strong solution to \eref{eqn.russo} has been also obtained in \cite{russo-trutnau}
under some technical assumption $\mathcal{A}(\nu_0)$ (see \cite{russo-trutnau}*{pg. 2229}). It is not clear whether this technical assumption can be verified for the equation~\eref{equation} which corresponds to~\eref{eqn.russo} with $\sigma=1$ and $b'=-\frac{1}{2}\dot W$.

The current paper offers the following contributions:  First,  we show that for any Brownian motion $B$, independent of $W$,  the It\^o-McKean representation  (\ref{sol-rep}) is a weak solution of the
equation~\eref{equation};  second,  for any given Brownian motion
$\cB$  we construct a particular Brownian motion $B$, independent of $W$, such that
the It\^o-McKean representation  (\ref{sol-rep}) is a strong  solution of the
equation~\eref{equation};  third, we show  
the strong uniqueness 
%(which implies the weak uniqueness) 
%\footnote{Kle did we show this?}
of the solution; 
%\footnote{Kle added}
and finally, we develop an It\^o calculus for the solution.  Note that the regularity of the generalized drift $b'=-\frac{1}{2} \dot W$ (where $W$ is H\"older
continuous with exponent $\alpha$, for any $\alpha$ less than $1/2$) is
at the border of what the papers mentioned above handled to show that  $X$ is a solution
of the stochastic differential equation with generalized drift.  While proving our results,
a major task for us   is to give a meaning to the integral 
$\int_0^t  \dot W(X(s))\,ds$ appearing in \eqref{equation}
and its approximations.  We shall complete this task by  using  the local time 
of a Brownian  motion and the following identity:  
\[
\int_0^t \dot W(X(s)) ds =\int_\RR e^{-W(x)} L_B(\xi, S_W(x))
W(dx)\Big|_{\xi=T_{W,B} ^{-1}(t)}\,.
\]
[See \eqref{key-identity} in the next section.]\ However, 
due to the lack of martingale property  of $L_B(\xi, y)$ on $\xi$,
we need to use Garsia-Rodemich-Rumsey theorem in order to give a meaning to the above object. This in turn forces 
us to study the higher moment properties of the local time of Brownian
motion, which has its own interest.  Let us also point out that 
our approach is probabilistic and we crucially  use the fact that $W$ is a Brownian motion.
In comparison with the results obtained in the  aforementioned papers,
the other results can be applied to (almost) every
sample path of $W$, but need to assume that $W$ has a H\"older continuity higher than $1/2$,
which cannot be verified by a Brownian motion. Our result can be applied to   Brownian motion
but is not for every sample path.

% \footnote{KL thinks only martingale aspect of $W$ is needed (to define stochastic integral in \eqref{Straton_int}). In particular, we did not use independence of increments of $W$! Is it correct?}

\textbf{Notations:}
Throughout
% \footnote{LM3 added notation for filtration}
the paper we will use a number of different filtrations and $\sigma$-fields. Set $\mathcal{F}^B=\{\mathcal{F}^B_t\}_{t\geq 0}$
be the filtration generated by the Brownian motion $B$. We will also need the extended filtration $\mathcal{F}^{B,W}=
 \{\mathcal{F}^{B,W}_t\}_{t\geq 0}$
 given by
 $$ \mathcal{F}^{B,W}_t = \mathcal{F}^{B}_t \bigvee \sigma(W(x), x\in \RR),\;t\geq 0.  $$

$C_b(\RR)$ denotes the space of all bounded continuous functions on $\RR$. For $\lambda\in (0,1)$, and $a<b$, let $\| \cdot \|_{\lambda, [a,b]}$ the $\lambda$-H\"older norm for functions on $[a,b]$, that
is,
\begin{equation}
\label{holdnorm}
 \| f\|_{\lambda, [a,b]} \equiv \|f\|_{\infty,[a,b]}+ \sup_{x,y\in [a,b]}\frac{|f(x)-f(y)|}{|x-y|^{\lambda}}
\end{equation}
where $ \|\cdot\|_{\infty,[a,b]}$ is the supremum norm. Similarly $\| \cdot \|_{\lambda}$ will denote the $\lambda$-H\"older norm for functions
on $\RR$.
Let $C^{\lambda}([a,b])$ (resp. $C^{\lambda}$)
be the space of H\"older continuous functions $f$ on $[a,b]$ (resp. on $\RR$)
with $\| f \|_{\lambda, [a,b]}<\infty$ (resp. $\| f \|_{\lambda,\RR}<\infty$ ).
The notation $A\lesssim B$ means $A\le C B$ for some non-negative constant $C$.

% \section{Main results} % (fold)
% \label{sec:main_results}
%   Since the drift term in \eqref{equation} has no conventional meaning, our first task is to define it properly. Let $W_\pi$ be a polygon approximation of $W$, with mesh size $|\pi|$. We will adopt the following definition through out the paper.
%   \begin{defition}
%       A continuous process $(X(t),t\ge0)$ is a strong solution to equation \eqref{equation} if
%       \begin{itemize}
%           \item For every $T>0$, with probability one, $\sup_{t\le T}\int_0^T \dot{W}_\pi(X(s))ds$ is convergent as the mesh size $|\pi|$ shrinks to 0. We denote this limit as $\sup_{t\le T}\int_0^T \dot{W}(X(s))ds$
%           \item $X$ satisifes equation \eqref{equation} almost surely.
%        \end{itemize}
%   \end{defition}
% % section main_results (end)
\section{Main results} % (fold)
\label{sec:main_results}
% In this section, we state our main results and outline their proofs. Details are provided in the subsequent sections. 

It is evident that to understand equation \eqref{equation}, one should first properly define the drift term $\int_0^t \dot{W}(X(s))ds$. For a two-sided Brownian motion $W$, $\dot{W}$ is not a function but a distribution (generalized functions), this integral has no canonical meaning. However, if the process $X$ admits the It\^o-McKean presentation \eqref{sol-rep} for some Brownian motion $B$ independent of $W$, we can define this integral in such a way that the map $W\mapsto \int \dot{W}(X(s))ds$ is an extension of the integration on smooth functions, i.e $\int \dot{f}(X(s))ds$ for a regular function $f$. 

Let us now describe our method in more details by the following heuristic argument. We first fix $W$ and $B$, and adopt the following strategy. Let $L_X(t, x)$ be the local time of the process $X$ which is defined as the unique process  such that
\begin{equation}
\int_0^t f(X(s)) ds=\int_\RR   L_X(t, x) f(x) dx\,, \quad \forall \ t\ge 0  \quad \hbox{and} \quad \forall f\in C_b(\RR).\label{e.1.3}
\end{equation}
From the representation \eref{sol-rep}, we see that
\begin{equation}
L_X(t, x)=e^{-W(x)} L_B(T_{W, B}^{-1}(t), S_W(x))\,,\label{e.1.4}
\end{equation}
where $L_B(t, x)$ is the local time for Brownian motion $B$, $S_W$ and $T_{W,B}$ are
defined by \eref{sw} and \eref{twb}. Using  the definition \eref{e.1.3} of the local time, we formally write
\begin{equation}
\label{20_08_1}
\int_0^t \dot W(X(s)) ds=\int_\RR L_X(t, x) \dot W(x)dx=\int_\RR L_X(t, x) W(d^ox)\,.
\end{equation}
%\footnote{Kle change} 
A fundamental problem arises: in what sense should one interpret $W(d^ox)$, the above stochastic integral with respect to $W$? Note that
for fixed $t$, the process $x\mapsto L_X(t,x)$  is not
necessarily adapted, which is one of the difficulties.  
If $W$ were a smooth function the above integral would be the usual (pathwise) integral. Hence the last integral in \eqref{20_08_1} should be defined as the
(anticipative) Stratonovich stochastic integral so that %\footnote{Kle rephrase} 
the integrations in \eqref{20_08_1} are extensions of the classical setting of smooth functions. It turns out that with this interpretation, the process $X$ given by~\eref{sol-rep} will indeed solve~\eref{equation} (weakly). This  can also been seen from our approximation argument described in Section \ref{sec:prelim}.

Let us explain how the  Stratonovich integral
 $\int_\RR L_X(t, x) W(d^ox)$ can be defined rigorously. Presumably, one may use the  anticipative stochastic calculus  (\cite{np}) (with the help of Malliavin calculus)
to define this integral.  However, we immediately encountered  a difficulty to show the square integrability of $L_X(t, x)$.  Instead, we use \eref{e.1.4} and \eref{20_08_1} to formally write
\begin{align}
\int_0^t \dot W(X(s)) ds
&= \int_\RR L_X(t, x) W(d^ox)=\int_\RR e^{-W(x)} L_B(T_{W,B}^{-1}(t), S_W(x))
W(d^ox)\nonumber\\
&=\int_\RR e^{-W(x)} L_B(\xi, S_W(x))
W(d^ox)\Big|_{\xi=T_{W,B} ^{-1}(t)}\,.\label{key-identity}
\end{align}
%The expression on  the right hand side of \eref{key-identity}  enables  us to give a meaning to $\int_0^t  \dot W(X(s)) ds$. 
The
%\footnote{LM rephrase} 
precise definition of the expression on the right hand side of \eref{key-identity} will be given in this section, and eventually this will enable 
us to give a meaning to $\int_0^t \dot W(X(s)) ds$ (see Definition~\ref{def2.2} below).

%Furthermore, 
In fact,
%\footnote{LM rephrased: changed Furthermore to "in fact"}
 throughout the paper, we can consider a more general situation, namely the integral of the type 
\begin{equation}\label{type.itgr}
	\int_0^t g(X(s),W(X(s)))\dot W(X(s)) ds.	
\end{equation}
This generalization will later allow us to develop It\^o calculus on equation \eqref{equation} and obtain strong uniqueness result. Concerning the function $g$, we assume that $g:\RR\times\RR\to\RR$ is a \textit{deterministic} continuous function such that
    \begin{itemize}
        \item For every $x\in\RR$, the function $u\mapsto g(x,u)$ is continuously differentiable,
        \item For every $u\in\RR$, the functions $x\mapsto g(x,u)$ and $x\mapsto \partial_u g(x,u)$ are H\"older continuous of order $\lambda$ with $\lambda>1/2$.
     \end{itemize}
In addition, we assume that $g$ satisfies the analytic bounds
    \begin{equation}\label{cond.g1}
        \sup_{x\in K}|g(x,u)|\le c_1(K)e^{\theta |u|}
    \end{equation}
    and
    \begin{equation}\label{cond.g2}
        \sup_{x,y\in K}\frac{|g(x,u)-g(y,u)|}{|x-y|^\lambda}+\sup_{x,y\in K}\frac{|\partial_ug(x,u)-\partial_ug(y,u)|}{|x-y|^\lambda}\le c_2(K)e^{\theta |u|}
    \end{equation}
    for every $u\in\RR$ and compact interval $K$, where $\theta, c_1(K) $ and $c_2(K)$ are some positive constants.

    Note that for any fixed $\xi\geq 0$, the mapping $x\mapsto g(x,W(x)) L_B(\xi, S_W(x)), x\in \RR_+$ is {\it adapted} with respect to the filtration generated by
$\{W(z),\; z\in [0,x]\}_{x\geq 0}$. Similarly the mapping $x\mapsto g(x,W(x)) L_B(\xi, S_W(x)), x\in \RR_{-}$ is {\it adapted} with respect to the filtration generated by
$\{W(z),\; z\in [x,0]\}_{x\leq 0}$. To elaborate this point, we define
$$ \widetilde W(x)=W(-x), \;x\geq 0.$$
Let $  W(dx)$ and $ \widetilde W(dx)$ denote It\^o differentials. Then for any $a\leq b$, and continuous function $g$ on $\RR^2$, we define the It\^o integral
\begin{multline}
\label{def_Ito}
 \int_a^b  g(x,W(x)) L_B(\xi, S_W(x))W(dx) 
	\\= \left\{ \begin{array}{l} \int_a^b g(x,W(x)) L_B(\xi, S_W(x))
W(dx),\; {\rm if} \; 0\leq a\leq b\\
\mbox{}\\
            \int_0^{|a|}  g(x,W(-x)) L_B(\xi, S_W(-x))\widetilde W(dx) \\
            \mbox{}\\
             \;\;\; + \int_0^b  g(x,W(x)) L_B(\xi, S_W(x))
W(dx),\; {\rm if} \; a\leq 0\leq b,
  \\
  \mbox{}\\
   \int_{|b|}^{|a|}  g(x,W(-x)) L_B(\xi, S_W(-x)) \widetilde W(dx),\;{\rm if}\; a\leq b\leq 0.
                \end{array}
\right.
\end{multline}
Now   for any $a\leq b,\, \xi>0$,  and any continuous  function $g$ satisfying \eqref{cond.g1} and $\eqref{cond.g2}$, we define
 \begin{multline}
      \int_{a}^b g(x,W(x))L_B(\xi,S_W(x))W(d^ox)
     	:= \int_{a}^b g(x,W(x))L_B(\xi,S_W(x))W(dx)
      \\-\frac12 \int_{a}^{b}\partial_u g(x,W(x))L_B(\xi,S_W(x))dx\,, \label{strat_int}
\end{multline}
where $\int_{a}^b g(x,W(x))L_B(\xi,S_W(x))W(dx)$ is the It\^o stochastic integral defined in~\eref{def_Ito}. While the right hand side of \eqref{strat_int} is valid for bigger classes of functions, we restricted ourselves to conditions \eqref{cond.g1} and \eqref{cond.g2} because it is this specific class in which most of the limiting results of the current work hold. The following result, whose proof %can be found 
is given in Section~\ref{sec:stratonovich_integral_II}, confirms that the integration defined in \eqref{strat_int} is indeed of Stratonovich type.

    \begin{proposition}\label{prop.strat} Assume that $g$ satisfies the conditions \eqref{cond.g1} and \eqref{cond.g2} with some  $\lambda>1/2$. In addition, we assume that $u\mapsto\partial_u g(x,u)$ is continuously differentiable. Fix arbitrary $a<b$. Let $\pi:  a=x_0<x_1<\cdots<x_n=b$ be  a partition of the interval $[a,b]$ and  let $|\pi|=\max_{0\le i\le n-1} (x_{i+1}-x_i)$.  Let
        \begin{equation}
            \label{polygW_f}
              W_\pi(x)=W(x_i)+\left(W(x_{i+1})-W(x_i)\right)\frac{x-x_i
              }{x_{i+1}-x_i}\,, \quad x_{i}\le x<x_{i+1}\,,
        \end{equation}
         be  the linear interpolation of $W$ associated with the partition $\pi$. Then
        \begin{multline}\label{def.Gab}
            \int_a^b g(x,W(x)) L_B(\xi,S_W(x))W(d^ox)
            \\=\lim_{|\pi|\to0} \int_a^b  g(x,W(x))L_B(\xi,S_W(x))\dot{W}_\pi(x)dx\,,
         \end{multline}
        where the limit in \eqref{def.Gab} is in $L^2$.
  \end{proposition}
\noindent The regularity of this integration is described in the following result.
% , whose proof is provided in Section~\ref{sec:pf.moment}.
  \begin{lemma}\label{prop.Gxiy} Let $g$ be a continuous function satisfying \eqref{cond.g1} and \eqref{cond.g2}. Then there exists a version of the process $$(\xi,a)\mapsto \int_{-a}^a g(x,W(x))L_B(\xi,S_W(x))W(d^ox)$$ which is jointly continuous in $(\xi,a)\in \RR_+\times \RR_+$.
  \end{lemma}
\begin{proof}
 From \eqref{strat_int}, it is sufficient to show the process $$H(\xi,y)=\int_{0}^{y}g(x,W(x))L_B(\xi,S_W(x))dW(x)$$ has a jointly continuous version. Fix $y_1<y_2<N$, $\xi_1<\xi_2$, using martingale moment inequality and \eqref{cond.g1}, we obtain
    \begin{align*}
      &\EE |H(\xi_1,y_1)-H(\xi_1,y_2)-H(\xi_2,y_1)+H(\xi_2,y_2)|^4 
      \\&= \EE \left|\int_{y_1}^{y_2}g(x,W(x))L_B([\xi_1,\xi_2],S_W(x))dW(x) \right|^4\\
      &\lesssim   |y_2-y_1|\int_{y_1}^{y_2}\EE e^{4 \theta|W(x)|}|L_B([\xi_1,\xi_2],S_W(x))|^4dx\,.
    \end{align*}
    It is straightforward to verify that (see also the identity \eqref{jmoment} below)
    \begin{equation*}
      \EE^B |L_B([\xi_1,\xi_2],S_W(x))|^4\le C|\xi_2-\xi_1|^2\,.
    \end{equation*}
    Hence,
    \begin{multline*}
      \EE |H(\xi_1,y_1)-H(\xi_1,y_2)-H(\xi_2,y_1)+H(\xi_2,y_2)|^4
      \\\le C|y_2-y_1||\xi_2-\xi_1|^2 \int_{y_1}^{y_2}e^{8 \theta^2|x|}dx 
      \le C_N |y_2-y_1|^2|\xi_2-\xi_1|^2\,.
    \end{multline*}
    The result then follows from two-parameter Kolmogorov theorem.
\end{proof}

\noindent As an immediate consequence, we have 
\begin{lemma}
\label{exist_G}
  Let $g$ be a continuous function satisfying \eqref{cond.g1} and \eqref{cond.g2}. Then for any fixed $\xi\geq 0$,  the limit
$$ \lim_{a\rightarrow\infty}  \int_{-a}^ag(x,W(x))L_B(\xi,S_W(x))W(d^ox)$$
exists almost surely. We will denote the limiting process as
\begin{equation*}
    \int_{-\infty}^\infty g(x,W(x))L_B(\xi,S_W(x))W(d^ox)\,.
\end{equation*}
Furthermore, for any fixed  $\xi\geq 0$, we define
\begin{equation}
    \tau_{W, B}
    (\xi)=\inf\{x>0:S_W(x)>|\max_{s\in[0,\xi]}B_s| \}\,. \label{tau-W-B}
\end{equation}
Then,
\begin{equation}
    \label{tau_fin}
    \tau_{W, B}(\xi)<\infty, \;{\rm a.s.},
\end{equation}
  and for all $\xi\ge0$,
\begin{multline}\label{id.ginfty}
\int_{-\infty}^\infty g(x,W(x))L_B(\xi,S_W(x))W(d^ox)
\\= \int_{-\tau_{W, B}(\xi)}^{\tau_{W, B}(\xi)} g(x,W(x))L_B(\xi,S_W(x))W(d^ox)\,.
\end{multline}
As a consequence, the process $\xi\mapsto \int_{-\infty}^\infty g(x,W(x))L_B(\xi,S_W(x))W(d^ox)$ has a continuous version.
\end{lemma}
\begin{proof}
    We denote $M_B(\xi)=|\max_{s\in[0,\xi]}B_s|$. A result of Matsumoto and Yor in \cite{matyor2}*{identity (4.5)} shows that
 \begin{equation}
\label{eq100}
\lim_{K\to\infty} \sqrt{2\pi K}\,\EE[S_W(K)^{-1}]=1\,.
\end{equation}
    On the other hand, for each $K>0$ (recall also that $B$ and $W$ are independent)
    \begin{align*}
      P(\tau_{W,B}(\xi)>K)
      &=P(S_W(K)^{-1}\ge M(\xi)^{-1})
      \\&\le \EE [M_B(\xi)]\EE [S_W(K)^{-1}]
      \lesssim   \EE [S_W(K)^{-1}]\,.
    \end{align*}
  Together with \eref{eq100}, it follows that $\lim_{K\to\infty} P(\tau_{W,B}(\xi)>K)=0$. From here, we deduce \eref{tau_fin}.

Since $S_W(\cdot)$ is strictly increasing, if $y$ is such that $y>\tau_{W,B}(\xi)$, then $S_W(y)> |\max_{s\in[0,\xi]}B_s|$, and hence $L_B(\xi,S_W(y))$ vanishes. As a consequence, with probability one, the map $x\mapsto g(x,W(x))L_B(\xi,S_W(x))$ is supported in the interval $[-\tau_{W,B}(\xi),\tau_{W,B}(\xi)]$. Therefore, the limit of $\int_{-a}^a g(x,W(x))L_B(\xi,S_W(x))W(d^ox)$ as $a$ goes to $\infty$ exists almost surely. From here, we also obtain \eqref{id.ginfty}. By Lemma~\ref{prop.Gxiy}, the map  $(\xi,a)\mapsto \int_{-a}^a g(x,W(x))L_B(\xi,S_W(x))W(d^ox)$ is continuous. This together with  continuity of
 $\xi\mapsto \tau_{W,B}(\xi)$ implies that the process $$\xi\mapsto\int_{-\infty}^\infty g(x,W(x))L_B(\xi,S_W(x))W(d^ox)$$ has a continuous version.
 \end{proof}
  %  We notice that $\tau(\cdot)$ is continuous, thus from Proposition \ref{prop.Gxiy}, $G(\cdot)$ has a continuous version. Moreover, since $S_W(\cdot)$ is %increasing, if $y$ is such that $y>\tau(\xi)$, then $S_W(y)> |\max_{s\in[0,\xi]}B_s|$, and hence $L_B(\xi,S_W(y))$ vanishes. Therefore, with probability %one, the function $x\mapsto g(x) e^{-W(x)}L_B(\xi,S_W(x))$ has support contained in $[-\tau(\xi),\tau(\xi)] $. This leads to the following presentation of %$G(\xi)$
  % \begin{equation}\label{def.G}
  %   G(\xi)=\int_{-\infty}^{\infty}g(x)e^{-W(x)}L_B(\xi,S_W(x))W(d^\circ x)\,.
  % \end{equation}
%Note also that for any continuous $g$ we can set
%$$ G(g,\xi,\infty) := G(g,\xi, -\tau_{W, B}(\xi), \tau_{W, B}(\xi)), \forall \xi \geq 0, $$
%where
% the local time $L_B(\xi, S_W(x))$ has compact support
%inside $[-\tau_{W, B}(\xi), \tau_{W, B}(\xi)]$. To simplify notation we denote
%$$  G(g,\xi) :=  G(g,\xi,\infty),\; \forall \xi\geq 0.$$

With the help of Lemmas~\ref{prop.Gxiy}, \ref{exist_G}  we can now
define the integral of the type \eqref{type.itgr} for sufficiently regular functions $g$ 
and $X$ as in \eqref{sol-rep}.
  \begin{definition}
\label{def2.2}
    Let $X$ be the process in \eqref{sol-rep}. Suppose that $g$ is a function satisfying conditions \eqref{cond.g1} and \eqref{cond.g2}. Then for every $t\ge0$, we define
    \begin{multline}\label{def.intW}
      \int_0^t g(X(s),W(X(s)))\dot{W}(X(s))ds
      \\:=\int_{-\infty}^\infty g(x,W(x))e^{-W(x)}L_B(\xi,S_W(x))W(d^ox)\big|_{\xi=T_{W,B} ^{-1}(t)},
    \end{multline}
    where $T_{W,B}$ is defined by \eref{twb} and $T_{W,B}^{-1}$ is the inverse
     of $T_{W,B}$.
In particular,  for $g\equiv 1$ we have
\begin{equation}\label{equ.st-integral}
\int_0^t \dot W(X(s)) ds=\int_{-\infty}^\infty e^{-W(x)}L_B(\xi,S_W(x))W(d^ox)\big|_{\xi=T_{W,B} ^{-1}(t)}
\end{equation}
for all $t\ge0$.
\end{definition}
\noindent From Lemma \ref{exist_G}, the process $\xi\mapsto \int_{-\infty}^\infty g(x,W(x))e^{-W(x)}L_B(\xi,S_W(x))W(d^ox)$
 has a continuous version. In addition, since the map $t\mapsto T^{-1}_{W,B}(t)$ is also continuous, we see
 that the process $$t\mapsto   \int_0^t g(X(s),W(X(s)))\dot{W}(X(s))ds$$ also has a continuous version. 
 From now on, we will only consider this continuous version whenever we write either
 $\int_0^t g(X(s),W(X(s)))\dot{W}(X(s))ds$ or alternatively its two other equivalent presentations
 \begin{multline*}
 	\int_{-\infty}^\infty g(x,W(x))e^{-W(x)}L_B(T^{-1}_{W,B}(t),S_W(x))W(d^ox)
 	\\=\int_{-\infty}^\infty g(x,W(x)) L_X(t,x)W(d^ox)\,.
 \end{multline*}
In the above, the equality can be seen from \eqref{e.1.4}.

% Recalling the notation in~\eref{strat_int}, the integrals in the above definition can also be written as follows
% \begin{equation}
% \int_0^t g(X(s)) \dot W(X(s)) ds=G(g, T_{W,B} ^{-1}(t))=
% \int_{-\infty}^\infty  g(x) e^{-W(x)} L_B(\xi, S_W(x))
%  W(d^\circ x)\Big|_{\xi=T_{W,B} ^{-1}(t)}\label{equ.st-integral_g}
% \end{equation}
% See Section 4
%for detailed discussion about this.
%In particular for $g\equiv 1$ we immediately get  from~\eref{equ.st-integral_g}
% \begin{equation}
% \int_0^t \dot W(X(s)) ds=G(T_{W,B} ^{-1}(t))=
% \int_{-\infty}^\infty  e^{-W(x)} L_B(\xi, S_W(x))
%  W(d^\circ x)\Big|_{\xi=T_{W,B} ^{-1}(t)}\,. \label{equ.st-integral_1}
% \end{equation}

Now with a rigorous definition of $\int_0^t \dot W(X(s)) ds$ at hand we can now precisely describe the notions of strong and weak solutions to~\eqref{equation}.
\begin{definition}[Strong solution] \label{def.strslt} 
Let $(W(x), x\in \RR)$ be a two-sided Brownian motion, and $(\cB(t), t\ge 0)$  be a Brownian motion with respect to a usual filtration $(\cF^\cB)_{t\ge0}$, independent of $W$. Let $\cF^{\cB,W}=(\cF^{\cB,W}_t)_{t\ge0}$ be the extended filtration given by
\begin{equation*}
 	\cF^{\cB,W}_t=\cF^{\cB}_t\bigvee \sigma(W(x),x\in\RR)\,,\forall t\ge0\,.
\end{equation*}
We assume that $\cF^{\cB,W}$ also satisfies the usual conditions. 
A continuous process $(X(t), t\geq 0)$ is a strong solution to~\eqref{equation} if it  satisfies the following conditions:
\begin{enumerate}[\rm(i)]
	\item $X$ is adapted to the extended filtration $\cF^{\cB,W}$.
 	\item There exists a Brownian motion $(B(t), t\ge 0)$  independent of $W$ such that $X(t)$ admits the It\^o-McKean representation \eref{sol-rep}.
	\item For every $t$, the integral $\int_0^t \dot{W}(X(s))ds $ is well defined as in Definition~\ref{def2.2}.        
    \item For every $t\ge0$, the equation
        \begin{equation*}
            X(t)=\cB(t)-\frac12\int_0^t \dot{W}(X(s))ds
         \end{equation*}
         holds almost surely. 
\end{enumerate}
\end{definition}
\begin{definition}[Weak solution]\label{def.wslt}
	Let $(W(x), x\in \RR)$ be a two-sided Brownian motion on a probability space $(\Omega, \cF, (\cF_t)_{t\geq 0}, P)$.
 A pair $(X,\cB)$ 
%in which $X$ is a continuous process, $\cB$ is a Brownian motion independent of $W$, 
is a weak solution to \eqref{equation} on $(\Omega, \cF, (\cF_t)_{t\geq 0}, P)$ if it  satisfies the following conditions:
% if $X$ is a strong solution to \eqref{equation}. More precisely, let $\cF^{\cB,W}$
% be the filtration defined as in Definition \ref{def.strslt}. $X$ satisfies the following conditions:
\begin{enumerate}[\rm(i)]
	\item $X$  is a continuous process process adapted to $ (\cF_t)_{t\geq 0}$ and $\cB$ is an $ (\cF_t)_{t\geq 0}$-Brownian motion 
 independent of $W$. 
 %is adapted to the filtration $\cF^{\cB,W}$.
 	\item There exists a Brownian motion $(B(t), t\ge 0)$  independent of $W$ such that $X(t)$ admits the It\^o-McKean representation \eref{sol-rep}.
	\item For every $t\ge0$, the integral $\int_0^t \dot{W}(X(s))ds $ is well defined as in Definition~\ref{def2.2}.        
    \item For every $t\ge0$, the equation
        \begin{equation*}
            X(t)=\cB(t)-\frac12\int_0^t \dot{W}(X(s))ds
         \end{equation*}
         holds almost surely. 
\end{enumerate}
\end{definition}
The major contribution of the current paper is the strong existence and uniqueness result for the Brox equation \eqref{equation}.
\begin{theorem} [Existence and uniqueness of strong solution] \label{thm.sext}
        Let $W$ be a two-sided Brownian motion and $\cB$ be a Brownian motion independent of  $W$.   
        Then there exists a unique strong solution $X$ to~\eqref{equation}. 
\end{theorem}
\noindent In proving Theorem \ref{thm.sext}, we are able to obtain existence of a pair $(X,\cB)$ satisfying \eqref{equation}. The precise statement is following. 
\begin{proposition} [Existence of a weak solution] \label{thm.wext}
	Let $(W(x), x\in \RR)$ be a two-sided Brownian motion and let
	$(B(t), t\ge 0)$ be a Brownian motion, independent of $W$.
	Let $X(t)$ be the It\^o-McKean representation given by the equation
	\eref{sol-rep} and let
	$\int_0^t \dot W(X(s)) ds$ be defined by \eref{equ.st-integral}.  Then, there is
	a Brownian motion $\mathcal{B}$  determined  by
	\begin{equation}\label{eqn.cb}
	\mathcal{B}(t)=\int_0^t e^{-W\circ S_W^{-1}\circ B\circ T_{W,B}^{-1}(s) }  d B\circ T^{-1}_{W, B}(s)\,,
	\end{equation}
	which is  independent of $W$,  such $(X,\cB)$ is a weak solution to equation \eqref{equation}.
\end{proposition}
\noindent In fact, Theorem~\ref{thm.wext} claims a bit more than just weak existence. It states  that any Brox diffusion given by the It\^o-McKean representation \eqref{sol-rep} is a weak solution to the equation \eqref{equation}.  In addition, the Brownian motion $\cB$ appeared in the equation is given explicitly by the equation \eref{eqn.cb}.

As an application of our method, we can easily obtain the following It\^o formula whose proof is provided in Section \ref{sec:Ito}.
\begin{theorem}[It\^o formula]\label{thm.ito}
Let $(X,\cB)$ be a weak solution to~\eqref{equation}. Let $f:\RR\times\RR\to \RR$ be a deterministic continuous function such that
\begin{itemize}
    \item For every $x$, the map $u\mapsto f(x,u)$ is continuously differentiable
    \item $f$ and $\partial_u f$ satisfy the conditions \eqref{cond.g1} and \eqref{cond.g2}.
 \end{itemize}
 We define the function $F(x)=\int_0^x f(y,W(y))dy+F(0)$, where $F(0)$ is some constant.
Then, with probability one,
  \begin{multline*}
    F(X(t))=F(0)+\int_0^t f(X(s), W(X(s))d\mathcal{B}(s)+\frac12\int_0^t \partial_xf(X(s),W(X(s)))ds
    \\-\frac12 \int_{-\infty}^\infty f(x,W(x))L_X(t,x)W(d^ox)
    +\frac12 \int_{-\infty}^\infty \partial_u f(x,W(x))L_X(t,x)W(d^ox)\,.
  \end{multline*}
\end{theorem}
\noindent An immediate corollary  is the following
%\footnote{HL: follow LM's remark, we put the 1st theorem (in previous version) as a corollary without proof. Feel free to change it as you please }
\begin{corollary}[It\^o formula]\label{thm.ito.det}
Let $(X,\cB)$ be a weak solution to~\eqref{equation}.
Let $F:[0, \infty)\times \RR\to \RR$ be a measurable deterministic
function    which is continuously differentiable in $t$ and twice continuously differentiable in $x$.
Then, with probability one,
  \begin{align*}
    F(t, X(t))&=F(0, 0)+\int_0^t \partial _s F(s, X(s))ds+\int_0^t \partial _x F(s, X(s))d\mathcal{B}(s)
    \\&\quad+\frac12\int_0^t \partial_{xx} F(s,X(s))ds
    -\frac12 \int_{-\infty}^\infty \partial_x F(s, x) L_X(t,x)W(d^ox)\,.
  \end{align*}
\end{corollary}

\smallskip
The rest of this paper is organized as follows. In the next section, we provide some preliminaries and show how Proposition \ref{thm.wext} can be derived. Theorem \ref{thm.ito} is proved in Section \ref{sec:Ito}. The proof of Theorem \ref{thm.sext} is given in Section \ref{sec:existence_of_strong_solution}.
The proof of Proposition~\ref{prop.strat} is provided in
Section~\ref{sec:stratonovich_integral_II}. Proofs of some further technical results (described in Section \ref{sec:prelim}) are provided in Sections \ref{sec.pf3.4}, \ref{sec:pf.moment} and \ref{sec.convg}.

% section main_results (end)
\section{Preliminaries and proof of Proposition \ref{thm.wext}}
\label{sec:prelim}
We present in the current section some necessary results which will be used several times throughout our paper. Since Proposition~\ref{thm.wext} follows directly from these results, we provide its proof at the end of the section.  

Let $\mathcal{\widetilde F}=\{\mathcal {\widetilde F}_t\}_{t\geq 0}$ be a filtration under which $B$ is a Brownian motion. We assume that the filtration
$\mathcal{\widetilde F}$ satisfies the usual conditions for a filtration; namely,
 it is right-continuous
and $\mathcal{\widetilde F}_0$ contains all the null sets. In what follows, $\mathcal{\widetilde F}$ is usually chosen to be $\cF^{B,W}$.
%\footnote{KL slight change of wordings}

%\footnote{KL modified this paragraph, added a new terminology "time change" \\  LM4 %also changed slightly the 2nd sentence}
An \textit{$\{\mathcal{\widetilde F}_t\}$-time-change} is a c\`adl\`ag, increasing family of $\{\mathcal{\widetilde F}_t\}$-stopping times. It is said to be \textit{finite} if each stopping time is finite almost surely, and \textit{continuous} if it is almost surely continuous with respect to time.
% Let $T=\{T(t):t\ge 0\}$ be a finite continuous $\{\mathcal{\widetilde F}_t\}$-time change. That is, for each $t\ge0$, $T_t$ is a finite $\mathcal{\widetilde F}$-stopping time. We assume that  continuous and increasing
%    in $t\ge 0$.   We assume that
% each stopping time $T(t)$ is finite almost surely and is adapted to
% the filtration $\mathcal{\widetilde F}$.
Let $T=\{T(t):t\ge 0\}$ be a finite $\{\mathcal{\widetilde F}_t\}$-time change and consider the time-changed filtration
$\{\mathcal{\widetilde F}_{T_t}\}_{t\geq 0}$. The right-continuity of $\{\mathcal{\widetilde F}_t\}$ and $\{T_t\}$ imply that $\{\mathcal{\widetilde F}_{T_t}\}_{t\geq 0}$ satisfies the usual conditions. %\footnote{LM3 Is it statement or assumption??\\KL added brief explanation}
Moreover, the
time-changed process $\{B\circ T(t)\}$ is an $\{\mathcal {\widetilde F}_{T_t}\}$-semimartingale (see \cite{jacod}*{Corollary 10.12}).
%\footnote{LM3 Is it statement? If so, why not a part  of the
%proposition?
%\newline
%They are statement. The proposition is already long. But if you prefer to put it
%in the proposition it is also fine.\\ KL: This is a property, I added reference }
 % with respect
% to the filtration $\{\mathcal {\widetilde F}_{T_t}\}_{t\geq 0}$.
As a consequence, one can define the It\^o
integral of the form $\int_0^t g(B\circ T(s) )dB\circ T(s) $.  In the following proposition we
gather some useful facts.
%\footnote
%{LM reference?
%
%Reply:  Given after the proposition}.
\begin{proposition}
\label{prop:3.1}
     Let $f$ be a function in  $C^ 2(\RR)$, the set of continuous functions with continuous derivatives up to second order.
  Let $T=(T(t),t\ge 0)$ be a continuous finite time change.
  % be a family of
  %  stopping times, which is continuous and increasing
  %  in $t\ge 0$.
   Then, with probability one, for all $t\ge0$, the following identities hold
%~\footnote{LM assumptions on $T(t)$?
%
 %  Reply:  Assumption was given before the proposition. Now it is repeated in the
  % proposition.  }
    \begin{equation}\label{id.ito}
        f(B\circ T(t))=f(B\circ T(0))+\int_{T(0)}^{T(t)}f'(B(s))dB(s)+\frac12\int_{T(0)}^{T(t)}f''(B(s))ds \,,
    \end{equation}
    \begin{equation}\label{id.ito2}
        \int_{T(0)}^{T(t)}f'(B(u))dB(u)=\int_0^t f'(B\circ T(s)) dB\circ T(s)\,,
    \end{equation}
    \begin{equation}\label{id.ito3}
        \int_{T(0)}^{T(t)}f''(B(u))du =\int_0^t f''(B\circ T(s)) dT(s) \,.
    \end{equation}
    Finally, the process $t\mapsto\int_0^t f'(B\circ T(s))dB\circ T(s) $ is a semimartingale with respect to the filtration
    $\{\mathcal {\widetilde F}_{T_t}\}_{t\geq 0}$,%\footnote{LM style changes}
and
its quadratic variation is given by
    \begin{equation}\label{id.qvbt}
        \langle\int_0^t f'(B\circ T(s))dB\circ T(s)\rangle= \int_0^t |f'(B\circ T(s))|^2dT(s)\,.
    \end{equation}
    % where
    % \[
    % \int_0^t f'(B\circ T(s)) dB(T(s))=\sum_{k=0}^{n-1}   f'(B(T(t_k)))(B(T(t_{k+1}))-B(T(t_{k})))\,.
    % \]
\end{proposition}

In fact, in \cite{kobay}, the author has obtained time-changed It\^o formula (such as \eqref{id.ito}) for semimartingales possibly with jumps. However, we do not need such general result in the current paper.  We refer the reader to \cite{kobay}*{Theorem 3.3} for a justification of \eqref{id.ito} and \eqref{id.qvbt}. Identities \eqref{id.ito2} and \eqref{id.ito3} follow from \cite{jacod}*{Proposition 10.21}, see also in \cite{kobay}.

Throughout the paper, we will approximate $W(x)$ by its polygonal approximations. Since $W(x)$ is defined for all $x\in \RR$, we now partition the whole line $\RR$.
Let $\pi$ be any partition with nodes $\{x_i\in\RR :
x_i<x_{i+1} \ \ \forall i\in\ZZ\}$. Then the polygonal approximation of $W$
associated with this partition, denoted by $W_\pi$, is the piecewise function such that for every $i\in\ZZ$
\begin{equation}
\label{polygW}
W_\pi (x)=W(x_i)+\dfrac{W(x_{i+1})-W(x_{i})}{x_{i+1}-x_i}
\left(x-x_i\right)\,, \quad x_{i}\le x<x_{i+1}\,.
\end{equation}

Fix arbitrary Brownian motion $B$ independent of $W$. Then, for any
polygonal  approximation $W_\pi$ of $W$,  we can define $X_{\pi}$ via an analogue to the It\^o-McKean representation \eref{sol-rep}:
\begin{equation}
X_\pi(t)=S_{W_{\pi}}^{-1} \circ B\circ T_{W_{\pi}, B}^{-1}(t)\,, \quad 0\le
t<\infty\,,\label{sol-rep-pi}
\end{equation}
where
\begin{equation}
S_{W_\pi} (x)=\int_0^x e^{W_\pi(z) } dz,\ \label{s_n}\,, \quad 0\le
t<\infty,
\end{equation}
and
\begin{equation}
T_{W_\pi, B}(t)=\int_0^t e^{-2W_\pi\circ S_{W_\pi} ^{-1}(B(s) )} ds\,, \quad 0\le
t<\infty.
\label{t_n}
\end{equation}
%\footnote{KL added this}
We also denote
\begin{equation}\label{eqn.cbpi}
\mathcal{B}_\pi(t)=\int_0^t e^{-W_\pi(X_\pi(s) )}  d B\circ T^{-1}_{W_\pi, B}(s)\,,
\end{equation}

% \footnote{LM3 added}
% Throughout the  proof of Proposition~\ref{thm.wext} we will show that $\cB$ is indeed
% Brownian motions independent of $W$.
Since $W_\pi$ is piecewise differentiable it follows from Proposition \ref{prop:3.1} that
\begin{lemma}
\label{pi_rep}%\footnote{KL added formula for $\cB_\pi$}
Let $X_\pi(t)$ be defined by
\eref{sol-rep-pi}-\eref{t_n} and $\cB_\pi(t)$ be defined in \eqref{eqn.cbpi}. Then $\cB_\pi$ is a Brownian motion
with
% \footnote{LM3 added filtration}
respect to the time-changed filtration $\{\mathcal{F}^{B,W}_{T^{-1}_{W_\pi, B}(t)}\}_{t\geq 0}$. In addition, $\cB_\pi$ is
 independent of $W$
 %\footnote{LM2
% I added independence: I hope this is trivial? KL: this is now shown in section \ref{sec.app} }
and $X_\pi$ satisfies
\begin{equation}
 X_\pi(t)= -\frac12\int_0^t \dot W_\pi(X_\pi(s)) ds +
 \mathcal{B}_\pi(t)\,.\label{e.2.16}
\end{equation}
\end{lemma}
\begin{proof}
The  It\^o formulas \eqref{id.ito}-\eqref{id.ito3} from~Proposition~\ref{prop:3.1} give
\begin{multline}
X_\pi(t)
=\int_0^t (S_{W_\pi}^{-1})'\circ B\circ T^{-1}_{W_\pi, B}(s) d B\circ T^{-1}_{W_\pi, B}(s)
\\ +\frac12 \int_0^t (S_{W_\pi}^{-1})''\circ B\circ T^{-1}_{W_\pi, B}(s) \frac{d}{ds} T^{-1}_{W_\pi, B}(s)ds\,.\label{e.3.6}
\end{multline}
Note
% \footnote{LM3 added; everywher throughaout the proof I mention that
%  $B$ is Br. motion wrt $\cF^{B,W}$}
 that we apply
 % the It\^o formula from
 Proposition~\ref{prop:3.1} by, first,
fixing a  realization of $W$; we also use the
fact that $B$ is a Brownian motion with respect to $\cF^{B,W}$.
From the definition of  $S_{W_\pi}(x)$, we have
\[
\frac{d}{dx} S_{W_\pi} ^{-1}(x)=e^{-W_\pi(S_{W_\pi}^{-1}(x))}\,,
\]
\[
\frac{d^2}{dx^2} S_{W_\pi} ^{-1}(x)=-e^{-2W_\pi(S_{W_\pi}^{-1}(x))}\dot W_\pi(S_{W_\pi}^{-1} (x))\,.
\]
Thus
\[
\left[ \frac{d}{dx} S_{W_\pi} ^{-1}\right]\circ B\circ T_{W_\pi, B}^{-1}(s) =e^{-W_\pi(X_\pi(s) )}\,,
\]
\[
\left[\frac{d^2}{dx^2} S_{W_\pi} ^{-1}
\right]\circ B\circ T_{W_\pi, B}^{-1}(s)   =-e^{-2W_\pi(X_\pi(t))}\dot W_\pi(X_\pi(s) )\,.
\]
Similarly, we have
\[
\frac{d}{dt} T_{W_\pi, B}^{-1}(s)= e^{2W_\pi(S_{W_\pi}^{-1}\circ B\circ T_{W_\pi, B}^{-1}(s))}=e^{2W_\pi(X_\pi(s))}\,.
\]
Thus \eref{e.3.6} can be written as
\begin{eqnarray}
X_\pi(t)
&=&\int_0^t e^{-W_\pi(X_\pi(s) )}  d B\circ T^{-1}_{W_\pi, B}(s)\nonumber\\
&&\qquad +\frac12 \int_0^t -e^{-2W_\pi(X_\pi(t))}\dot W_\pi(X_\pi(s) )e^{2W_\pi(X_\pi(s)} ds \nonumber\\
\label{approx_eq1}
&=&\mathcal{B}_\pi(t)-\frac12 \int_0^t  \dot W_\pi(X_\pi(s) )  ds\,.
\end{eqnarray}

From  Doob's optional stopping (sampling)
theorem it is easy to see that
%\footnote{LM3 cannot we just cite something?
%Proposition~\ref{prop:3.1}? Or it does not follow ffr4om it?
%\newline
%The sentence is slightly changed which is more precise I think. KL thinks it is implicit in proposition \ref{prop:3.1} \\
%LM4 I am OK with this; the sentence is slightly changed}
 % From the previous section we see that
 $(\mathcal{B}_\pi(t),t\ge0)$ is a local
 % \footnote{LM3 I added ``local'' since it does not look that one can get more from
 % the argument above? Also added filtration}
 martingale
 % \footnote{LM2 where do we see this from??}
 with respect to $\{\mathcal{F}^{B,W}_{T^{-1}_{W_\pi, B}(t)}\}_{t\geq 0}$.
 % \footnote{KL: slight change of notation \\ LM3 also changed}
% $\cG_t=\cF_{T_t}$
Moreover, its  quadratic variation is
\begin{align}
\label{eq3_07_2}
\int_0^t e^{-2W_\pi(X_\pi(s) )} \frac{d}{ds}  T^{-1}_{W_\pi, B}(s)ds= \int_0^t e^{-2W_\pi(X_\pi(s) )} e^{2W_\pi(X_\pi(s))}ds= t\,.
\end{align}
Thus by L\'evy's characterization theorem
  $\mathcal{B}_\pi(t)$ is a Brownian motion with respect to $\{\mathcal{F}^{B,W}_{T^{-1}_{W_\pi, B}(t)}\}_{t\geq 0}$.

% \footnote{KL: added the following}
To complete the proof of Lemma~\ref{pi_rep}, it remains to show that $\cB_\pi$ and $W$ are independent processes. Since both of them are Gaussian, it suffices to show that they are uncorrelated. Indeed, using \eqref{id.ito2} and \eqref{sol-rep-pi}, we can write
\begin{align*}
    \cB_\pi(t)=\int_0^{T^{-1}_{W_\pi,B}(t)}e^{-W_\pi\circ S^{-1}_{W_\pi} \circ B(u)}dB(u)\,.
\end{align*}
Hence, for every $t\ge0$ and $x\in\RR$, we use the fact that $\cB_{\pi}$ is
$\{\mathcal{F}^{B,W}_{T^{-1}_{W_\pi, B}(t)}\}_{t\geq 0}$-Brownian motion, and the fact that
 $W$ is measurable with respect to $
 \mathcal{F}^{B,W}_{T^{-1}_{W_\pi, B}(0)}
 %=\mathcal{F}^{B,W}_{0}=\mathcal{F}^{ W} \qquad (\hbox{since $T^{-1}_{W_\pi, B}(0)=0$})
$
 to get
\begin{align*}
    \EE \left[\cB_\pi(t)W(x)\right]=
    \EE \left[\EE \left[\cB_{\pi}(t)
    %\int_0^{T^{-1}_{W_\pi,B}(t)}e^{-W_\pi\circ S^{-1}_{W_\pi} \circ B(u)}dB(u)
    \left|\mathcal{F}^{B,W}_{T^{-1}_{W_\pi, B}(0)}\right.\right]
 W(x)\right] =
 \EE \left[ \cB_\pi(0)
 W(x)\right]
 =0.
\end{align*}
%because $\langle B,W\rangle_z=\langle B,\widetilde W\rangle_z=0$ for every $z\ge0$.
% This together with~\eref{approx_eq1} finishes\footnote{LM2 We have to prove independence of
 % $\mathcal{B}_{\pi}$ of $W$???}  the proof of Lemma~\ref{pi_rep}.
Hence, we complete the proof of Lemma \ref{pi_rep}.
\end{proof}

% We conclude this section with the following remarks.
% \footnote{LM2 put the part of the proposition as remark}
%irst of note that the lemma implies that
\begin{remark}\label{rem.32} (i) Lemma~\ref{pi_rep} implies that
$(X_\pi(t), t\ge 0)$ is the weak solution of the
equation: $$dX_\pi(t)=-\frac12 \dot W_\pi(X_\pi(t))dt+d\widetilde{B}(t),$$
where $\widetilde{B}$ is a Brownian motion independent of $W_{\pi}$.

(ii) The result of Lemma \ref{pi_rep} holds  true when $W_\pi(x)$ is
replaced by any continuously  differentiable function.
\end{remark}
%\begin{equation}
%X_\pi(t)= -\frac12 \int_0^t \dot W_\pi(X_\pi(s)) ds+  \mathcal{B}_t^\pi\,, \quad 0\le
%t<\infty.\label{e.x-pi}
%\end{equation}
Now Proposition~\ref{thm.wext} follows from \eqref{e.2.16} by shrinking the mesh size $|\pi|$ to 0. This step is verified through the following propositions.
\begin{proposition}\label{prop.bpib} For every $T\ge0$, $\lim_{|\pi|\to0} \EE \sup_{t\le T}|\cB_\pi(t)-\cB(t)|^2=0$.
\end{proposition}
 \begin{proposition}\label{prop.convg}
 Then for every $\delta>0$, there exists a partition $\pi(\delta)$ of  $\RR$ such that for any $T>0$,
    \begin{multline*}
      \lim_{\delta\to0}\sup_{t\leq T}\left|\int_0^t g(X_{\pi(\delta)}(s),W_{\pi(\delta)} \circ X_{\pi(\delta)}(s))\dot{W}_{\pi(\delta)}(X_{\pi(\delta)}(s))ds\right.
      \\\left.-\int_{-\infty}^\infty g(x,W(x))L_X(t,x)W(d^ox) \right|=0,
    \end{multline*} with probability one.
    % Here $\int_0^t g(X(s))\dot{W}(X(s))ds$ is defined via~\eref{def.intW}.
\end{proposition}
\noindent The proofs of the above two propositions are provided in Section \ref{sec.pf3.4} and Section \ref{sec.convg} respectively.
Proposition~\ref{prop.convg} in turn is relied on the following moment estimates for local time of Brownian motion, which are of independent interest.
\begin{proposition}\label{Prop.moment} (i) Let $x,y\in \RR$. For every   $\beta\in[0,1/2]$,
  the following estimates holds
    \begin{equation}
\left|\EE
\left(L_B([\xi,\eta],y)-L_B([\xi,\eta],x)\right)^{2n}\right|
      \le C_{\beta,n}|\eta-\xi|^{n(1-\beta)}|x-y|^{2\beta
      n} \label{ineq.Lxy}\,.
\end{equation}
(ii)  For every $x_1, y_1, \cdots\,, x_k, y_k$ satisfying
 \begin{equation}\label{cond:xyk}
    x_1<y_1\le x_2<y_2\le \cdots\le x_{2n}<y_{2n}\,.
  \end{equation}
and every $\alpha\in[0,1]$ we have
 % for any $0\le \al \le 1$,
\begin{equation}
\left|\EE \prod_{k=1}^{2n}
\left(L_B([\xi,\eta],y_k)-L_B([\xi,\eta],x_k)\right)\right|
      \le C_{\alpha,n} |\eta-\xi|^{n \alpha}\prod_{k=1}^{2n}|y_k-x_k|^{1- \alpha}\,.
      \label{ineq.Lxyk}
\end{equation}
\end{proposition}
\noindent The proof of the previous proposition is given in Section \ref{sec:pf.moment}.
  \begin{remark} (i) The former inequality \eref{ineq.Lxy} is well known. The above
second estimate \eref{ineq.Lxyk}  is new and quite interesting
itself. Since in our proof of \eref{ineq.Lxyk} we shall obtain  some
results which can be used to prove \eref{ineq.Lxy} easily, we shall
also present a straightforward proof of \eref{ineq.Lxy}.

(ii) From \cite{perkins}, it is known that $L(\xi, x)$ is a semimartingale 
on $x$.  A consequence is that 
$\EE \left(L_B(\xi,y)-L_B(\xi,x)\right)^{2n} 
      \le C_{\beta,n} |x-y|^{ 
      n}$. \eqref{ineq.Lxy} is an extension of this inequality. 
\end{remark}
We will also need the following analytic result.
\begin{lemma}
    \label{lemma_inverse}
    Let $f$ and $f_n$, ($n=1,2,...$) be bijective functions on $\RR$
    which are continuous and strictly increasing.
    Suppose that $f_n(x)$ converges to $f(x)$ for every $x$ in $\RR$.
    Then for any compact
    $A\subset \RR$, $\lim_{n\to\infty}\sup_{y\in A} |f_n^{-1}(y)-f^{-1}(y)| =0.$
  \end{lemma}
  \begin{proof}
    The proof follows by contradiction. Suppose there exists
    $\epsilon_0$ and a subsequences $\{f_{n_k}\}$ and $\{y_{n_k}\}$  such that
   \begin{equation*}
    y_{n_k}\rightarrow y, \; \mbox{\rm as}\; n_k\rightarrow\infty,
    \end{equation*}
    \begin{equation*}
      |f^{-1}_{n_k}(y_{n_k})-f^{-1}(y)|>\epsilon_0\,,\quad \forall n_k\,.
    \end{equation*}
    Thus, for infinitely many $n_k$'s, either
     $f_{n_k}^{-1}(y_{n_k})>f^{-1}(y)+\epsilon_0$ or $f_{n_k}^{-1}(y_{n_k})<f^{-1}(y)-\epsilon_0$.
     Without lost of generality, we consider only the former case
     in which $y_{n_k}>f_{n_k}(f^{-1}(y)+\epsilon_0 )$ for infinitely
     many $n_k$'s. Upon passing   the limit $n_k\to\infty$,
     we obtain $y\ge f(f^{-1}(y)+\epsilon_0)>f(f^{-1}(y))$, which is a contradiction.
  \end{proof}
 % {
%  \color{red} Does one needs to show $\lim_{t\to\infty}T_{W,B}(t)=\infty$
%  }
    % \begin{remark}
    % \footnote{KL added this remark, although unrelated to the whole paper. Should we keep it here?}
    %   One can apply Lemma~\ref{lemma_inverse} once more to the sequence $\{f_n^{-1}\}_n$ and $f^{-1}$. This yields $f_n$ converges to $f$ uniformly over compact sets. In this situation, pointwise convergence implies local uniform convergence under strict monotonicity condition on each $f_n$ and the limiting function. This phenomena is similar to Dini's convergence theorem whereas monotonicity of the sequence $f_n$ is imposed.
    % \end{remark}

Let us see how  Proposition \ref{thm.wext} follows from these propositions.
\smallskip

\noindent
\begin{proof}[Proof of Proposition \ref{thm.wext}]
By Proposition~\ref{prop.convg} (with $g\equiv 1$) we see that
$\int_0^t  W_\pi(X_\pi(s))
ds $ converges almost surely to
$\int_0^t \dot W(X(s)) ds$ uniformly in $t$ on compacts of $\RR_+$. It is also obvious  from the definitions of
$S_{W_\pi} (x)$, $T_{W_\pi, B}(t)$, $X_\pi(t)$,
and application
of Lemma~\ref{lemma_inverse}
that $X_\pi(t)$ converges almost surely to $X(t)$ uniformly in $t$ on compact intervals of $\RR_+$.
From Proposition~\ref{prop.bpib}, it follows that
$\mathcal{B}_\pi(t)$ converges almost surely  to the process $\cB$ defined in \eqref{eqn.cb} uniformly on  compact intervals of $\RR_+$. By passing through the limit $|\pi|\to0$ in \eqref{e.2.16}, we see that $X$ satisfies \eqref{equation}. In addition, by  Lemma~\ref{pi_rep}, for {\it every} $\pi$, $\mathcal{B}_\pi$ is the Brownian motion independent of $W$, hence, it is trivial to see that the limiting process
$\mathcal{B}$ is also  a  Brownian motion independent of $W$.  
This finishes the proof. 
%To conclude the  proof, it remains to show that $X$ is adapted to the filtration $\cF^{\cB,W}$. Indeed, because the integrand
% in \eqref{eqn.cb} is non-vanishing, it follows that the filtration $\cF^{\cB,W}$ coincides with the time-changed filtration $\cF^{B,W}_{T^{-1}_{W,B}}$. 
%In addition, it is evident that the process $X$ defined by It\^o-McKean representation \eqref{sol-rep} is adapted to the later filtration. These two facts %complete the proof.
\end{proof}

% section prelim (end)

\section{It\^o formula - Proof of Theorem \ref{thm.ito}}
\label{sec:Ito}
\setcounter{equation}{0}
% \textbf{Proof of Theorem \ref{thm.ito}:}
\begin{proof}[Proof of Theorem \ref{thm.ito}]
  Let $\pi$ be any partition of $\RR$. Let $W_\pi$ be the linear interpolation of $W$ defined  by  \eqref{polygW}.
 Denote  $F_\pi(x)=\int_0^x f(y,W_\pi(y))dy+F(0)$ and
 $X_\pi(t) =  S^{-1}_{W_\pi}\circ B\circ T^{-1}_{W_\pi,B}(t)$.  We apply the time-changed It\^o formula \eqref{id.ito}
  for $F_\pi(X_\pi(t))=F_\pi\circ S^{-1}_{W_\pi}\circ B\circ T^{-1}_{W_\pi,B}(t)$,
  recall that in order to apply It\^o formula we  first fix a  realization of $W$ and we also
  use the fact that $B$ is a Brownian motion with respect to $\cF^{B,W}$.
\begin{align*}
  F_\pi(X_\pi(t))&=F(0)+\int_0^t (F_\pi\circ S^{-1}_{W_\pi})'\circ B\circ T^{-1}_{W_\pi,B}(s)\,dB\circ T_{W_\pi,B}^{-1}(s)\\
  &\quad+\frac12 \int_0^t (F_\pi\circ S^{-1}_{W_\pi})''\circ B\circ T^{-1}_{W_\pi,B}(s)\, dT^{-1}_{W_\pi,B}(s) \,.
\end{align*}
It is now easy to see that
\begin{align*}
  (F_\pi\circ S^{-1}_{W_\pi})'\circ B\circ T^{-1}_{W_\pi,B}(s)=f(X_\pi(s),W_\pi(X_\pi(s)))e^{-W_\pi(X_\pi(s))}\,,
\end{align*}
\begin{align*}
  (F_\pi\circ S^{-1}_{W_\pi})''\circ B\circ T^{-1}_{W_\pi,B}(s)
  &=\partial_xf(X_\pi(s), W_\pi(X_\pi(s)))e^{-2W_\pi(X_\pi(s))}
  \\&\quad + \partial_uf(X_\pi(s), W_\pi(X_\pi(s)))e^{-2W_\pi(X_\pi(s))}\dot{W}_\pi(X_\pi(s))
  \\&\quad-f(X_\pi(s), W_\pi(X_\pi(s)))e^{-2W_\pi(X_\pi(s))}\dot{W}_\pi(X_\pi(s))\,,
\end{align*}
and
\begin{equation*}
  dT^{-1}_{W_\pi,B}(s)=e^{2W_\pi(X_\pi(s))}ds\,.
\end{equation*}
Upon combining the above four identities, we obtain
\begin{equation}\label{ito.Xpi}
	\begin{split}
		F_\pi(X_\pi(t))
  		&=F(0)+\int_0^t f(X_\pi(s), W_\pi(X_\pi(s)))\,d \mathcal{B}_\pi(s)
  		\\&\quad+\frac12\int_0^t \partial_x f(X_\pi(s),W_\pi(X_\pi(s)))ds
   		\\&\quad-\frac12 \int_0^t f(X_\pi(s),W_\pi(X_\pi(s)))\dot{W}_\pi(X_\pi(s))ds
   		\\&\quad+\frac12 \int_0^t f'(X_\pi(s),W_\pi(X_\pi(s)))\dot{W}_\pi(X_\pi(s))ds\,,
	\end{split}
\end{equation}
where $\mathcal{B}_\pi(t)=\int_0^t e^{-W_\pi(X_\pi(s))}dB\circ T_{W_\pi,B}^{-1}(s)$ is a Brownian motion, as seen from Lemma~\ref{pi_rep}. For every $\delta>0$, from Proposition \ref{prop.convg}, we can choose a partition $\pi=\pi(\delta)$ such that
  \begin{equation*}
    \int_0^t f(X_{\pi(\delta)}(s), W_{\pi(\delta)}(X_{\pi(\delta)}(s)))\dot{W}(X_{\pi(\delta)}(s))ds
  \end{equation*}
  and
  \begin{equation*}
    \int_0^t f'(X_{\pi(\delta)}(s), W_{\pi(\delta)}(X_{\pi(\delta)}(s)))\dot{W}(X_{\pi(\delta)}(s))ds
  \end{equation*}
  converge to 
  \begin{equation*}
  	\int_\RR f(x,W(x))e^{-W(x)}L_B(T^{-1}_{W,B}(t),S_W(x))W(d^ox)
  \end{equation*}
  and
  \begin{equation*}
  	\int_\RR f'(x,W(x))e^{-W(x)}L_B(T^{-1}_{W,B}(t),S_W(x))W(d^ox)
  \end{equation*}
  respectively as $\delta\downarrow0$. In addition, since $X_\pi$ and $W_\pi$ converge to $X$ and $W$, respectively, uniformly over compact intervals, with probability one, the integral $\int_0^t\partial_xf(X_\pi(s),W_\pi(X_\pi(s)))ds$ converges to $\int_0^t\partial_xf(X(s),W(X(s)))ds$. Hence, by passing through the limit $\delta\downarrow0 $ in \eqref{ito.Xpi}, it remains to show that the stochastic integral
$$\int_0^t f(X_\pi(s), W_\pi(X_\pi(s)))\,d \mathcal{B}_\pi(s)$$ converges
to $\int_0^t f(X(s), W(X(s)))d\mathcal{B}(s)$ in probability as the mesh size of $\pi$ shrinks to 0. For this purpose, we fix a continuous sample path of $W$ and further denote $\tilde f(x)=f(x,W(x))$ and $\tilde f_\pi(x)=f(x,W_\pi(x))$. Since for fixed $t>0$, $X_\pi$ converges uniformly to $X$ on $[0,t]$, for each $M>0$ we can find a stopping time $T_M$ such that
\begin{equation*}
    \sup_{s\le t}\sup_{\pi}|X_{\pi} (s\wedge T_M)|\le M \,.
\end{equation*}
Since $X$ has finite range, we can also require
$\lim_{M\to\infty}T_M=\infty$. Thus, it suffices to show the following limit in $L^2$
\begin{equation*}
    \lim_{|\pi|\downarrow0} \int_0^{t\wedge T_M}\tilde f_\pi(X_\pi(s))d\cB_\pi(s)=\int_0^{t\wedge T_M}\tilde f(X(s))d\cB(s)\,.
\end{equation*}
Similarly to the proof of Proposition~\ref{prop.bpib}, it is equivalent to show
\begin{multline}
\label{eq3_07_5}
    \lim_{|\pi|\to0} \EE^B \left[\int_0^{t\wedge T_M} \tilde f_\pi(X_\pi(s))\,d \mathcal{B}_\pi(s)\int_0^{t\wedge T_M} \tilde f(X(s))\,d \mathcal{B}(s)\right]
    \\=\EE^B\int_0^{t\wedge T_M} |\tilde f(X(s))|^2\,d s\,.
\end{multline}

Indeed, by the It\^o isometry, the expectation on the left side equals to
\begin{align*}
    \EE^B \left[\int_0^{T_{W_\pi,B}^{-1}(t\wedge T_M)\wedge T_{W,B}^{-1}(t\wedge T_M)} (\tilde f_\pi\circ S_{W_\pi}^{-1})'\circ B(u)\cdot(\tilde f\circ S_{W}^{-1})'\circ B(u)\,du\right]\,.
\end{align*}
It follows from Lemma~\ref{lemma_inverse} that with probability one
\begin{multline*}
    \lim_{|\pi|\to0}\int_0^{T_{W_\pi,B}^{-1}(t\wedge T_M)\wedge T_{W,B}^{-1}(t\wedge T_M)} (\tilde f_\pi\circ S_{W_\pi}^{-1})'\circ B(u)\cdot(\tilde f\circ S_{W}^{-1})'\circ B(u)du
    \\=\int_0^{T_{W,B}^{-1}(t\wedge T_M)} |(\tilde f\circ S_{W}^{-1})'\circ B(u)|^2\,du=\int_0^{t\wedge T_M} |\tilde f(X(s))|^2d s \,.
\end{multline*}
As in  the proof of  Proposition~\ref{prop.bpib} we can use the Cauchy-Schwarz inequality and some changes of variables to see that
\begin{align*}
& \left(\int_0^{T_{W_\pi,B}^{-1}(t\wedge T_M)\wedge T_{W,B}^{-1}(t\wedge T_M)}
(\tilde f_\pi\circ S_{W_\pi}^{-1})'\circ B(u)\cdot(\tilde f\circ S_{W}^{-1})'\circ B(u)\,du\right)^2
    \\&\le\int_0^{T_{W_\pi,B}^{-1}(t\wedge T_M)} |(\tilde f_\pi\circ S_{W_\pi}^{-1})'\circ B(u)|^2du
     \int_0^{T_{W,B}^{-1}(t\wedge T_M)}|(\tilde f\circ S_{W}^{-1})'\circ B(u)|^2du
    \\&=\int_0^{t\wedge T_M} |\tilde f_\pi(X_\pi(s))|^2ds \int_0^{t\wedge T_M} |\tilde f(X(s))|^2ds
    \\&\le t^2 \sup_{|x|\le M} |\tilde f(x)|^4\,.
\end{align*}
We may  use uniform integrability to get~(\ref{eq3_07_5}) and then
to conclude the proof. 
\end{proof}
%As an application we compute the quadratic variation of $\EE^W X$.
%\begin{corollary}
%    $\EE^W X $ is a supper martingale with quadratic variation
%    \begin{equation*}
%        t+??
%    \end{equation*}
%\end{corollary}
%\begin{proof}
%    Apply It\^o formula for $F(x)=x^2$ we obtain
%    \begin{align*}
%        X^2(t)=2\int_0^tX(s)d\cB(s)+t-\int_{-\infty}^\infty xL_X(t,x)W(d^ox)\,.
%    \end{align*}
%    Let $Y$ be the process $Y=X\circ T_{W,B}$. Then
%    \begin{equation*}
%        Y^2(t)=2\int_0^{T_{W,B}(t)}X(s)d\cB(s)+T_{W,B}(t)-\int_{-\infty}^\infty xe^{-W(x)} L_B(t,S_W(x))W(d^ox)
%    \end{equation*}
%    We note that
%    \begin{align*}
%        \int_{-\infty}^\infty xe^{-W(x)}L_B(t,S_W(x))W(d^ox)&=\int_{-\infty}^\infty xe^{-W(x)} L_B(t,S_W(x))W(dx)-\frac12\int_{-\infty}^\infty xe^{-W(x)} L_B(t,S_W(x))dx
%    \end{align*}
%    {\color{red}From here, can we say that $\EE^WY$ has quadratic variation
%    \begin{equation*}
%        \EE^WT_{W,B}(t)+\frac12\EE^W\int_{-\infty}^\infty xe^{-W(x)} L_B(t,S_W(x))dx
%    \end{equation*}
%    and maybe compute quadratic variation of $\EE^W X$
%      }
%\end{proof}

\section{Strong solution - Proof of Theorem~\ref{thm.sext}} % (fold)
\label{sec:existence_of_strong_solution}
\subsection{Existence part of Theorem~\ref{thm.sext}}
\label{sub:existence}
Because the methods proving existence and uniqueness are quiet different, we consider them separately. In this subsection, we focus on showing existence of a strong solution to equation \eqref{equation}. Throughout the current section, $W$ is a (given)  two-sided Brownian motioin and $\cB$ is a
(given) Brownian motion independent of  $W$.
 We first seek for a Brownian motion $B$ such that relation \eqref{eqn.cb} is verified.
  For this purpose, we first prove the following result.
\begin{lemma}\label{lem.m}
    Let $\cB$ be a Brownian motion and  let $W$ be two-sided  Brownian motion independent of $\cB$.
    Then, for $P$-a.s. $W$, the equation
    \begin{equation}
        M  (t) =\int_0^t    e^{W \circ S_W^{-1} \circ M( u)}  d\cB
        (u)\,, \quad t\ge 0\label{e.M}
    \end{equation}
    has unique strong solution $( M(t), t\ge 0)$  which has continuous sample paths.
\end{lemma}
\begin{proof}
First, we show the
existence of the weak solution to \eref{e.M}. In fact,
let $\tilde B$ be a Brownian independent from $W$. We define
$$\tilde\cB(t)=\int_0^t e^{-W\circ S_W^{-1}\circ \tilde B\circ T_{W,\tilde B}^{-1}(s)}
d\tilde B\circ T_{W,\tilde B}^{-1}(s)\,.$$  Then,  it follows from Proposition \ref{thm.wext} that
$\tilde\cB(t)$ is a Brownian motion, independent of $W$. Denote  $\tilde M=\tilde B\circ T_{W,\tilde B}^{-1} $.
Then $d\tilde\cB(t)=  e^{-W\circ S_W^{-1}\circ M(t)}
dM(t) $ or $
dM(t) =e^{W\circ S_W^{-1}\circ M(t)}  d\tilde\cB(t)  $.  This means that
 $(\tilde M,\tilde B)$  is a weak solution to equation \eqref{e.M}.

Let us prove the pathwise uniqueness for equation\eref{e.M}. Note that by the classical L\'evy theorem
$W$ satisfies the following modulus of continuity  condition: for each $n\geq 1$,
\begin{eqnarray}
|W(x,\omega)-W(x',\omega)|
&\leq& c_n(\omega)\log(|x-x'|)\sqrt{|x-x'|} \qquad  \forall x,x'\in [-n,n] \;\,, \nonumber   \\
&&\qquad \text{for some } c_n(\omega)\geq 0,\; \text{for } P-\text{a.s. } \omega.
\end{eqnarray}
Thus we can find  a set $A\subset \Omega$ with $P(A)=1$, such that, for all $\omega\in A$, the following holds: for any $n\ge 1$, there exists
$c_n(\omega)\ge  0$, such that
\begin{equation*}
|W(x,\omega)-W(x',\omega)|\leq \rho_n(x,x'),\;\;\forall x,x'\in [-n,n]\,,
\end{equation*}
where $\rho_n(x,x'):=c_n(\omega)\log(|x-x'|)\sqrt{|x-x'|}$. Fix arbitrary $\omega\in A$.   For any $k\geq 1$, we define
\begin{equation}
\phi_k(z)=\phi_k(z,\omega)=   e^{W(S_W^{-1}(-k\vee (z\wedge k)),\omega)}
\end{equation}
and consider the following stochastic differential equation
\begin{eqnarray}
M_k(t) &=&\int_0^t   \phi_k(M_k( u))  d\cB
(u)\,.\label{e.M.n}
\end{eqnarray}
Note that
\begin{align}
\nonumber
\int_{0+}^1 (\sqrt{|\log(u)u|})^{-2})\,du &= -\int_{0+}^1 (\log(u))^{-1}\,d(\log u)\\
&=  \int_{1}^{\infty} \frac{1}{v}\,dv%\\
%\label{eqrho}
%&=&
=\infty. \label{eqrho}
\end{align}
We now take
$$ n(k,w) = \lfloor |S_W^{-1}(k)| + |S_W^{-1}(-k)|+1 \rfloor
\,,$$
where $\lfloor a\rfloor$ denotes the integer part of $a$.
Then
\begin{equation*}
|W(x,\omega)-W(x',\omega)|\leq c_n(\omega)\log(|x-x'|)\sqrt{|x-x'|}
\end{equation*}
for all $x,x'$ in the interval $[-n(k,w),n(k,w)]$.
Define
\begin{eqnarray*}
S^*(\omega)=\sup_{|x|\leq  |S_W^{-1}(-k)|+  |S_W^{-1}(k)|} ( e^{W(x, \om)}+e^{-W(x, \om)}) .
\end{eqnarray*}
Then
\begin{align*}
 |\phi_k(z)-\phi_k(z')|
 &\leq S^* |W(S_W^{-1}(-k\vee (z\wedge k)))-W(S_W^{-1}(-k\vee (z'\wedge k)))|\\
&\leq   S^* \rho(|S_W^{-1}(-k\vee (z\wedge k))-S_W^{-1}(-k\vee (z'\wedge k))|).
\end{align*}
Note that $S_W^{-1}$ is Lipschitz function and we can easily derive:
$$|S_W^{-1}(-k\vee (z\wedge k))-S_W^{-1}(-k\vee (z'\wedge k))|\leq S^*|z-z'|,$$
and hence
$$ |\phi_k(z)-\phi_k(z')|
\leq S^*  \rho(S^*|z-z'|).
$$
This  together with \eref{eqrho}  implies the pathwise
uniqueness of the equation  \eref{e.M.n}
by standard Yamada-Watanabe criterion (see  \cite{iw}, Chapter IV, Theorem 3.2).

Now, let $M^1$ and $M^2$ be two continuous solutions to \eref{e.M}.
Define the following stopping times:
\begin{align*}
T^{M_1,W}_{k}&= \inf\{t\geq 0:\; M^{1}(t)=S_{W}(k)\; \text{or}\;  M^{1}(t)=S_{W}(-k )\},\\
T^{M_2,W}_{k}&= \inf\{t\geq 0:\; M^{2}(t)=S_{W}(k)\; \text{or}\;  M^{2}(t)=S_{W}(-k  )\},\\
\widetilde{T}^{ W}_{k}& = \min\left(T^{M_1,W}_{k}, T^{M_2,W}_{k}\right).
\end{align*}
Since the processes $(M^{1}(t), t\ge 0)$ and $(M^{2}(t), t\ge 0)$ have
continuous sample paths, we see $\widetilde{T}^{ W}_{k}\uparrow \infty$ a.s. when $k\rightarrow
\infty$. When $t\le \widetilde{T}^{ W}_{k}$, both $(M^{1}(t), t\ge 0)$ and $(M^{2}(t), t\ge 0)$ satisfy
\eref{e.M.n}. Thus $ M^{1}(t)=M^{2}(t)$ when $t\le \widetilde{T}^{ W}_{k}$. Passing through the limit $k\rightarrow \infty$ yields the strong uniqueness of the equation \eref{e.M}.

Finally, because weak existence and strong uniqueness together imply strong existence, we see
that the equation \eref{e.M} has a unique strong solution.
\end{proof}

We are now ready to prove the existence part of Theorem~\ref{thm.sext}.
\begin{proof}[Proof of existence part of Theorem~\ref{thm.sext}]
Let $M$ be the unique strong solution to equation \eqref{e.M}.
% found in Lemma~\ref{lem.m}.
Define a stopping $\tau(t)$ so that
\begin{equation}\label{def.tauM}
    \int_0^{\tau(t)}e^{2W\circ S_W^{-1}\circ M(s)}ds=t\,.
\end{equation}
%\footnote{Kle add this to explain $\tau$ is unique}
 We note that if $M$ and $W$ are provided, $\tau$ is uniquely determined by \eqref{def.tauM} because the map $u\mapsto \int_0^u e^{2W\circ S_W^{-1}\circ M(s)}ds$ is strictly increasing on $\RR_+$. We define $B=M\circ \tau$.  It follows from \eqref{e.M} that
\begin{align*}
    \langle B\rangle_t=\int_0^{\tau(t)}e^{2W\circ S_W^{-1}\circ M(s)}ds=t\,.
\end{align*}
Thus, from L\'evy's characterization theorem, $B$ is a Brownian motion.
In addition, the relation \eqref{def.tauM} is equivalent to
\begin{equation*}
    \tau(t)=\int_0^t e^{-2W\circ S_W^{-1}\circ M\circ \tau(s)}ds\,.
\end{equation*}
Hence, taking into account the relation $M\circ \tau=B$, we
have 
\begin{equation}
\label{eq:28_4} 
\tau(t)=\int_0^t e^{-2W\circ S_W^{-1}\circ B(s)}ds=T_{W,B}(t).
\end{equation}
 From here and the equation \eqref{e.M} it follows that $B$ and $\cB$ satisfy the relation \eqref{eqn.cb}.
In addition, similar to
the proof of Proposition \ref{thm.wext} it is clear that $B$ is independent of
$W$.

We now define $X=S_W^{-1}\circ B\circ T_{W,B}^{-1}$. Then Proposition~\ref{thm.wext}
shows that $X$ is a weak solution to \eqref{equation} since we have shown that $\cB$ and $B$ satisfy the relation  \eqref{eqn.cb}.
Now by~\eqref{eq:28_4}  we get that 
\begin{equation}
\label{eq:28_4_1}
 X=S_W^{-1}\circ B\circ \tau^{-1}=S_W^{-1}\circ M,
\end{equation}
where the last equality follows by the definition of $B$. Since $M$ is the unique strong solution to  \eqref{e.M}, we get that $M$ is adapted to filtration 
 $\cF^{\cB,W}$, and hence by \eqref{eq:28_4_1} $X$  is also adapted to filtration 
 $\cF^{\cB,W}$. This finishes the proof  that $X$ is a strong solution to 
the equation \eref{equation}.
\end{proof}
% section existence_of_strong_solution (end)

\subsection{Uniqueness part of Theorem~\ref{thm.sext}} % (fold)
\label{sub:strong_uniqueness}
To show uniqueness for strong solutions of \eqref{equation}, we rely on It\^o formula, Theorem \ref{thm.ito}.
\begin{proof}[Proof of uniqueness part of Theorem~\ref{thm.sext}]
Let $\cB$ be a Brownian motion independent of $W$. We would like to show that $X$ constructed in the 
proof of the existence part of Theorem~\ref{thm.sext} is indeed the unique 
strong solution to the equation~\eqref{equation}. Let $\wt X$ be another strong  solution, and let $\wt B$ the corresponding Brownian motion in the It\^o-McKean 
representation, that is 
\begin{eqnarray}
\label{eq:im_a}
 \wt X=S_W^{-1}\circ \wt B\circ T_{W,\wt B}^{-1}.
\end{eqnarray}
Here, as usual, 
\begin{equation}
T_{W, \wt B}(t)=\int_0^t e^{-2W\circ S_W^{-1}(\wt B(s) )} ds\,,\label{twb1}
\end{equation}
or alternatively $T_{W, \wt B}(t)$ satisfies 
 \begin{equation}\label{def.tilde_TWX}
        \int_0^{T_{W,\wt B}(t)}e^{2W\circ \wt X(s)}ds=t\,.
    \end{equation}
The advantage of the later definition is that $T_{W, \wt B}(t)$ is given only via $\wt X$. By a simple transformation one can see that $\wt B$ can be expressed via $\wt X$ as 
   \begin{equation}
\label{eq:un_wB1}
         \wt B(t)=S_{W}\circ \wt X\circ T_{W,\wt B}.
  \end{equation}
Now we would like to express $\wt B$ as a solution to certain stochastic equation driven by $\cB$. To this end we apply It\^o formula from  Theorem \ref{thm.ito}  to the function $S_W(x)=\int_0^x e^{W(y)}
    dy$. However, we cannot do it directly, since $x\mapsto e^x$ does not have bounded derivatives. Therefore an approximation is needed.
    Let $R$ be a fixed positive number. Let $f_R$ be a  $C^3$-function with bounded derivatives such
    that $f_R(x)=e^x$ for every $x\in[-R,R]$ and $f_R=0$ outside $[-R-1,R+1]$. We then apply  It\^o formula from  Theorem \ref{thm.ito} 
    to the function $F_R(x)=\int_0^x f_R(W(y))dy$ to get
    \begin{align*}
        F_R(\wt X(t))&=\int_0^t f_R(W(\wt X(s)))d\cB(s)-\frac12\int_{-\infty}^\infty f_R(W(x))L_{\wt X}(t,x)W(d^ox)
        \nonumber\\&\quad+\frac12   \int_{-\infty}^\infty f'_R(W(x))L_{\wt X}(t,x)W(d^ox)\,.
    \end{align*}
    Since $\wt X$ has continuous sample paths, $L_{\wt X}(t,\cdot)$ vanishes outside of a compact interval
%\footnote{LM should we justify it? KLe: give a brief reason for this} 
 (independent from $R$). We can pass  easily to the limit, as  $R\to\infty$,  and obtain
    \begin{equation}
        S_W(\wt X(t))=\int_0^t e^{W(\wt X(s))}d\cB(s).\label{e.9.sw_a}
    \end{equation}
The previous equation, \eqref{eq:un_wB1} and  \eqref{eq:im_a} imply 
   \begin{align}
        \wt B(t) &=\int_0^{T_{W,\wt B}(t)} e^{W(\wt X(s))}d\cB(s)
\nonumber 
\\
 \label{e.9.sw_b}
&= \int_0^{T_{W,\wt B}(t)} e^{W\circ S_W^{-1}\circ \wt B\circ T_{W,\wt B}^{-1} (s)}d\cB(s)
    \end{align}
Then we immediately obtain
  \begin{equation*}
        \wt B(T^{-1}_{W,\wt B}(t)) =\int_0^{t} e^{W\circ S_W^{-1}\circ \wt B\circ T_{W,\wt B}^{-1}(s))}d\cB(s).
    \end{equation*}
Thus $(\wt B \circ T^{-1}_{W, \wt B}(t), t\ge 0)$ satisfies \eref{e.M}.
   However, Lemma \ref{lem.m} states that the   equation  \eref{e.M} has the unique strong solution. That is, if $M(t)=\wt B \circ T^{-1}_{W, \wt B}(t)$ then $M$ is uniquely determined from the equation \eref{e.M}. 
   % \footnote{Kle rephrase} 
    In addition, upon comparing \eqref{def.tilde_TWX} with \eqref{def.tauM}, we see that $T_{W, \wt B}(t)=\tau(t)$ where $\tau(t)$   
    is uniquely
%\footnote{LM do we have some reference to uniqueness? Kle: a sentence is added after \eqref{def.tauM} to explain this }
 defined by \eref{def.tauM}. Note that both $M$ and $\tau$ are solutions of equations (\eqref{e.M} and \eqref{def.tauM} respectively) which do
 not  depend on {\it particular} solution $\wt X$ for~\eqref{equation}.
Then we have 
\begin{align*} \wt B &= M\circ \tau\\
 &= B, \;\;{\rm a.s.} 
\end{align*}
where $B$ is the Brownian motion constructed in the proof
%\footnote{LM added "existence part of"} 
of the existence part of Theorem~\ref{thm.sext}.  This and \eqref{eq:im_a} imply that 
$$ \wt X =X,\;\;{\rm a.s.}$$
and uniqueness follows. 
\end{proof}
\section{Proof of Proposition~\ref{prop.strat}} % (fold)
\label{sec:stratonovich_integral_II}
\setcounter{equation}{0}
    We  have the following decomposition
    \begin{equation*}
      \int_a^b g(x,W(x))L_B(\xi,S_W(x))\dot{W}_\pi(x) dx =I_1+I_2+I_3+I_4\,
    \end{equation*}
    where
    \begin{align*}
      &I_1=\sum_{k=0}^{n-1} \int_{x_k}^{x_{k+1}}[g(x,W(x))-g(x_k,W(x))]L_B(\xi,S_W(x))\frac{W(x_{k+1})-W(x_k)}{x_{k+1}-x_k}dx\,,\\
      & I_2=\sum_{k=0}^{n-1} \int_{x_k }^{x_{k+1}}g(x_k,W(x))[L_B(\xi,S_W(x))-L_B(\xi,S_W(x_k))]\frac{W(x_{k+1})-W(x_k)}{x_{k+1}-x_k}dx\,,\\
      & I_3=\sum_{k=0}^{n-1}\int_{x_k }^{x_{k+1}}[g(x_k,W(x))-g(x_k,W(x_k))]L_B(\xi,S_W(x_k))\frac{W(x_{k+1})-W(x_k)}{x_{k+1}-x_k}dx\,,\\
      & I_4=\sum_{k=0}^{n-1} g(x_k,W(x_k))L_B(\xi,S_W(x_k))[W(x_{k+1})-W(x_k)]\,.
    \end{align*}
    From the Cauchy-Schwarz inequality we see that $I_1^2$ is at most
    \begin{align*}
      (b-a) \sum_{k=0}^{n-1} \int_{x_k}^{x_{k+1}} |g(x,W(x))-g(x_k,W(x))|^2      |L_B(\xi, S_W(x))|^2\frac{[W(x_{k+1})-W(x_k)]^2}{(x_{k+1}-x_k)^2}dx\,.
    \end{align*}
    Taking expectation and applying the H\"older inequality and \eqref{cond.g1} we obtain
    \begin{equation*}
      \EE I^2_1\lesssim \sum_{k=0}^{n-1}
      |x_{k+1}-x_k|^{2 \lambda}\lesssim  |\pi|^{2 \lambda-1}
    \end{equation*}
    which implies $\EE I_1^2$ goes to 0 since $\lambda>1/2$.

    Denote each term in the expression of $I_2$ by $I_{2k}$.  Then
    \begin{equation*}
    \EE(I_2^2)
    = \sum_{k=0}^{n-1} \EE(I_{2k}^2)
    + \sum_{k\neq j} \EE(I_{2k}I_{2j})=:I_{2,1}+I_{2,2}\,.
    \end{equation*}
    From the Cauchy-Schwarz inequality we see that $I_{2,1}$ is at most
    \begin{equation*}
    	\sum_{k=0}^{n-1}  \int_{x_k}^{x_{k+1}}\EE \left\{ |g(x_k,W(x))|^2 \left[ L_B(\xi, S_{W}(x )) -  L_B(\xi, S_{W}(x_k))\right]^2
      \frac{[W(x_{k+1})-W(x_k)]^2} { x_{k+1}-x_k  } \right\} dx
    \end{equation*}
    By conditioning on the $\si$-algebra generated by $W$
    (namely taking the expectation with respect to the Brownian motion $B$ first)
    and applying \eqref{ineq.Lxy} with $\beta=1/2$, we see that
    \begin{eqnarray*}
     I_{2,1}
     \lesssim \sum_{k=0}^{n-1}   \int_{ x_k}^{ x_{k+1}} \EE  |g(x_k,W(x))|^2 \left[\left| S_{W}(x )  -    S_{W}(x_k) \right|
        \frac{\left(W(x_{k+1})-W(x_k)\right)^2 } {x_{k+1}-x_k}  \right]dx.
    \end{eqnarray*}
    which is majorized by a constant multiple of $|\pi|$.  It follows that $\lim_{|\pi|\to0}I_{2,1}=0$. If $k\not=j$ and if $x\in [x_j, x_{j+1})$
    and $z\in [x_k, x_{k+1})$,  then
    the intervals $[S_{W}(x_j), S_{W}(x))$ and $[S_{W}(x_k), S_{W}(z))$
    are disjoint.  Then we have from \eqref{ineq.Lxyk} with $\alpha=0$,
    \begin{align*}
    &\EE[ L_B(\xi, S_{W}(x )) -  L_B(\xi, S_{W}(x_j))][ L_B(\xi, S_{W}(z ))
    -  L_B(\xi, S_{W}(x_k))]\\
  & \le \EE  |S_{W}(x )-S_{W}(x_j )| |S_{W}(z )-S_{W}(x_k )|\,.
    \end{align*}
    Therefore, together with \eqref{cond.g1}, we have
    \begin{align*}
     I_{2,2}
     &\lesssim  \sum_{j<k}   \int_{ x_j}^{ x_{j+1}}  \int_{ x_k}^{ x_{k+1}} \EE  \Bigg[
      e^{\theta |W(x )|+\theta|W(z)| } \left| S_{W}(x )  -    S_{W}(x_k) \right|
      \left| S_{W}(z )  -    S_{W}(x_k) \right|   \\
    &\quad    \left|  \frac{ W(x_{j+1})-W(x_j)  } {x_{j+1}-x_j}\frac{ W(x_{k+1})-W(x_k) }
       {x_{k+1}-x_k}\right|     \Bigg]dx dz  \,.
    \end{align*}
    It is now easy to check that $I_{2,2}$ converges to $0$, hence so does $I_2$.

    Using the Taylor expansion, we have \[g(x_k,W(x))-g(x_k,W(x_k)) =\partial_ug(x_k,W(x))
    (W(x)-W(x_k))+R_k(x)\] with $\sup_{0\le x\le y}\EE|R_k(x)|^p\le C_p |x_{k+1}-x_k|^p$. Hence, we can decompose $I_3=I_{3,1}+I_{3,3}+I_{3,3}$, where
    \begin{multline*}
    	I_{3,1}= \sum_{k=0}^{n-1} \left[\int_{ x_k}^{ x_{k+1}} (W(x)-W(x_k))dx \frac{W(x_{k+1})-W(x_k)}{x_{k+1}-x_k}-\frac{1}{2}(x_{k+1}-x_k) \right]
    	\\\times\partial_ug(x_k,W(x_k))L_B(\xi,S_W(x_k))\,,
    \end{multline*}
    \begin{equation*}
    	I_{3,2}=\frac{1}{2}\sum_{k=0}^{n-1} \partial_ug(x_k,W(x_k))L_B(\xi,S_W(x_k))( x_{k+1}- x_k)\,,
    \end{equation*}
    and
    \begin{equation*}
    	I_{3,3}=\sum_{k=0}^{n-1} \int_{ x_k}^{ x_{k+1}}R_k(x)dx\,L_B(\xi,S_W(x_k))\frac{W(x_{k+1})-W(x_k)}{x_{k+1}-x_k}\,.
    \end{equation*}
    $I_{3,1}$ is a  sum of martingale difference. It is easy to see that
    \begin{align*}
      \EE (I_{3,1})^2 &\le \sum_{k=0}^{n-1} \EE \left[|\partial_ug(x_k,W(x_k))|^2 L_B^2(\xi,S_W(x_k)) \right]\\
      &\quad \times \left[\int_{ x_k}^{ x_{k+1}} (W(x)-W(x_k))dx \frac{W(x_{k+1})-W(x_k)}
      {x_{k+1}-x_k}-\frac{1}{2}(x_{k+1}-x_k) \right]^2 \\
      &\le C \sum_{k=0}^{n-1} ( x_{k+1}-  x_k)^2\to 0\,.
    \end{align*}
    $I_{3,2}$ is the Riemann sum of the integral $\frac12\int_a^b \partial_ug(x,W(x))L_B(\xi,S_W(x))dx$. A straightforward estimation yields that $I_{3,3}$ converges to 0 in $L^2$. Hence, we have $I_3$ converges to $\frac12\int_a^b \partial_ug(x,W(x))L_B(\xi,S_W(x))dx$ in $L^2$.
    By standard It\^o calculus, we see that $I_4$ converges in $L^2$ to the It\^o integral $\int_{a}^{b}g(x,W(x))L_B(\xi,S_W(x))W(dx)$.\qed

\section{Proof of Proposition~\ref{prop.bpib}}\label{sec.pf3.4}
\setcounter{equation}{0}

% \footnote{LM3 I do not think that the Lemma 4.2 that was here is necessarily, since it seems to me the standard result
% that convergence in probability + uniform integrability implies convergence in $L^1$. So I deleted it. Agree?}
%\begin{lemma}\label{lem.znz}
%   Suppose $Z_n$ converges in probability to $Z$ as $n$ goes to infinity and
%   \begin{equation*}
%       \sup_{n}\EE |Z_n-Z|^2=M^2<\infty\,.
%   \end{equation*}
%   Then $\EE | Z_n-Z|$ converges to 0.
%\end{lemma}
%\begin{proof}
%   Assume there exists $\epsilon>0$ such that $\EE|Z_n-Z|>2\epsilon$ for infinitely many $n$. Applying Cauchy-Schwarz inequality, we obtain
%   \begin{align*}
%       2 \epsilon&<\EE[|Z_n-Z|1_{|Z_n-Z|>\epsilon}]+\EE[|Z_n-Z|1_{|Z_n-Z|\le\epsilon}]
%       \\&\le M [P(|Z_n-Z|>\epsilon)]^{1/2}+ \epsilon\,.
%   \end{align*}
%   Thus, $[P(|Z_n-Z|>\epsilon)]^{1/2}>\epsilon M^{-1}$ for infinitely many $n$, which is a contradiction.
%\end{proof}

% \textbf{Proof of Proposition~\ref{prop.bpib}.}
From Doob's maximal inequality, it suffices to show
\begin{equation*}
    \lim_{|\pi|\to0} \EE |\cB_\pi(t)-\cB(t)|^2=0\,,
\end{equation*}
for every fixed $t>0$.
We write
\begin{align*}
\mathcal{B}_\pi(t)
&=\int_0^t e^{-W_\pi(X_\pi(s) )}  d B\circ T^{-1}_{W_\pi, B}(s)\\
&=\int_0^{T^{-1}_{W_\pi, B}(t)} e^{-W_\pi(X_\pi \circ T^{-1}_{W_\pi, B}(u) ) }  d B(u)\\
&=\int_0^{T^{-1}_{W_\pi, B}(t)} e^{-W_\pi(S_{W_\pi, B}^{-1}\circ B(u) ) }  d B(u)
\end{align*}
and similarly
\begin{align*}
\cB(t)
=\int_0^{T^{-1}_{W, B}(t)} e^{-W(S_{W, B}^{-1}\circ B(u) ) }  d B(u)\,.
\end{align*}
By
% \footnote{LM3 I put quadratic variation of $\cB$ here, to be able to say that it is
% Brownian motion, hence both $\cB_{\pi}$ and $\cB$ are square integrable
% martingale, and hence we can apply formula for expectation of the
% product}
a change of variable (similar to the one used in the proof of Lemma \ref{pi_rep}), we can immediately get that  the quadratic variation of $\cB$ is given by
\begin{equation}
\label{eq3_07_3}
\int_0^{T^{-1}_{W, B}(t)} e^{-2W(S_{W, B}^{-1}\circ B(u) )}  d u
=\int_0^{t} e^{-2W(S_{W, B}^{-1}\circ B\circ T^{-1}_{W, B}(s) )}  d T^{-1}_{W, B}(s)=t\,,
\end{equation}
and hence  $\cB$
%\footnote{KL change $\cB_\pi$ to $\cB$}
is a Brownian motion with respect to $\{\mathcal{F}^{B,W}_{T^{-1}_{W, B}(t)}\}_{t\geq 0}$. %\footnote{KL change}
In addition,
Lemma~\ref{pi_rep} asserts that $\cB_{\pi}$ is a Brownian motion with respect to $\{\mathcal{F}^{B,W}_{T^{-1}_{W, B}(t)}\}_{t\geq 0}$.
Since $\cB_{\pi}$ and $\cB$ are square integrable martingales we get
\begin{align}
\nonumber
\label{eq3_07_4}
\EE\left(\mathcal{B}_\pi(t) \cB(t)\right)
&=\EE\left[
\int_0^{T^{-1}_{W_\pi, B}(t)} e^{-W_\pi(S_{W_\pi, B}^{-1}\circ B(u) ) }  d B(u)
\int_0^{T^{-1}_{W, B}(t)} e^{-W(S_{W, B}^{-1}\circ B(u) ) }  d B(u)\right]\\
&= \EE\left[
\int_0^{T^{-1}_{W_\pi, B}(t)   \wedge T^{-1}_{W, B}(t) }  e^{
-W_\pi(S_{W_\pi, B}^{-1}\circ B(u) )-W(S_{W, B}^{-1}\circ B(u) ) }  d u\right]\,.
\end{align}
From Lemma~\ref{lemma_inverse}, $T_{W_\pi,B}^{-1}$ and $S_{W_\pi}^{-1}$
converge uniformly over finite intervals, almost surely,
% \footnote{LM3 added a.s.}
 to $T_{W,B}^{-1}$ and $S_{W}^{-1}$
respectively. Hence, for each $t\geq 0$,
% \footnote{LM3 added $t\geq 0$}
\begin{multline}
\label{eq3_07_1}
\int_0^{T^{-1}_{W_\pi, B}(t)   \wedge T^{-1}_{W, B}(t) }  e^{
-W_\pi(S_{W_\pi, B}^{-1}\circ B(u) )-W(S_{W, B}^{-1}\circ B(u) ) }  d u
\\\rightarrow \int_0^{  T^{-1}_{W, B}(t) }  e^{ -2W(S_{W, B}^{-1}\circ B(u) ) }  d u=t,%\;\;{\rm almost surely.}
\end{multline}
with probability one, as $|\pi|\rightarrow 0$. The last equality follows from~(\ref{eq3_07_3}).

%On the other hand,
Now, by first applying  Cauchy-Schwarz inequality, and then  equalities (\ref{eq3_07_2}) and (\ref{eq3_07_3}) we
get
% \footnote{LM3 First I did not undertand why do you need epectation below, so Idelted it;
  % second in the end I did not understand why did you have $\leq t^2$. I got $=t^2$}
\begin{multline*}
\left(
\int_0^{T^{-1}_{W_\pi, B}(t)   \wedge T^{-1}_{W, B}(t) }  e^{-W_\pi(S_{W_\pi, B}^{-1}\circ B(u) )
-W(S_{W, B}^{-1}\circ B(u) ) }  d u\right)^2
    \\\le\int_0^{T^{-1}_{W_\pi, B}(t)} e^{-2W_\pi(S_{W_\pi, B}^{-1}\circ B(u) )}  d u
    \int_0^{T^{-1}_{W, B}(t)} e^{-2W(S_{W, B}^{-1}\circ B(u) )}  d u =  t^2\,.
\end{multline*}
%Applying Lemma~\ref{lem.znz},
The above bound implies
% \footnote{LM3 As I said I do not see why we need the former Lemma 4.2}
 uniform integrability of random variables
$$\int_0^{T^{-1}_{W_\pi, B}(t)   \wedge T^{-1}_{W, B}(t) }  e^{-W_\pi(S_{W_\pi, B}^{-1}\circ B(u) )
-W(S_{W, B}^{-1}\circ B(u) ) }  d u,$$
and hence by (\ref{eq3_07_1}) we get that the right hand side of~(\ref{eq3_07_4}) converges to $t$, and this immediately
implies
 that $\lim_{|\pi|\to0}\EE\cB_\pi(t)\cB(t)=t$. Therefore,
\begin{align*}
\EE(\mathcal{B}_\pi(t) -\cB(t))^2
&=
\EE(\mathcal{B}_\pi(t)^2) +\EE( \cB(t) ^2)-
2\EE(\mathcal{B}_\pi(t) \cB(t))\\
&=2t-
2\EE(\mathcal{B}_\pi(t) \cB(t))
\end{align*}
converges to $0$ as $|\pi|\to0$. \qed
\section{Proof of Proposition~\ref{Prop.moment}}
\label{sec:pf.moment}
\setcounter{equation}{0}
%\footnote{LM2 I did not check this section,
%although added proof of Lemma 2.1 in the end.}
% Let $L_B(\xi , x)$ be the local time of the Brownian motion. \footnote{LM added} Formally, we can represent the local time of the Brownian motion $B$ as %$L_B(\xi,x)=\int_0^\xi \delta(B_s-x)ds$.   Hence $$L_B([\xi,\eta],x)=\int_\xi^\eta \delta(B_s-x)ds.$$
%In this section,  we\footnote{LM derive better?} derive two estimates on the moments of local time $L_B(\xi, x) $.

%\begin{proof}
Let  $p(t,x)=\displaystyle \frac{1}{\sqrt{2\pi t}}e^{-\frac{x^2}{2t}}$ be the heat kernel and $\mathfrak{S}_m$ denote the symmetric group of permutations of $\{1,  2, \cdots, m\}$. It is easy to verify  that for generic points $u_1,\dots,u_m$ in $\RR$, we have
  \begin{equation}\label{jmoment}
    \EE \prod_{j=1}^m L_B([\xi,\eta],u_j)=\sum_{\sigma\in\mathfrak{S}_m} \int_{D_m} \prod_{j=1}^m p(s_j-s_{j-1},u_{\sigma_j}-u_{\sigma_{j-1}})\,d \bar s\,,
  \end{equation}
  where $D_m$ is the domain $\{\bar s \in [\xi,\eta]^m:\xi<s_1<\cdots<s_m<\eta \}$,
  $d\bar s=ds_1\cdots ds_m$,  and $u_0=0$ by convention. \eqref{jmoment} is in fact the so-called Kac moment formula (see Marcus-Rosen's book \cite{marcus}).

To use  \eref{jmoment} to compute the two moments in \eqref{ineq.Lxy} and \eqref{ineq.Lxyk}, we need to introduce  some  notations. As introduced in \cite{hl}, for $k=1,\dots,n$ and $x\in\RR$, $V_k(x)$ denotes the substitution operator, i.e. for a generic function $f=f(u_1\,, \cdots\,, u_n)$, $V_k(x)f(u)=f(u_1, \cdots, u_{k-1}, x, u_{k+1}, \cdots, u_m)$.
It is clear that if $f(u)$ is a random process, then
\[
\EE V_k(x) f(u)=\EE   f(u_1, \cdots, u_{k-1}, x, u_{k+1}, \cdots, u_m)
=V_k(x) \EE   f(u)\,.
\]
Thus the operator $V_k$ commutes with the expectation operator.

For any points $x_1, \cdots, x_m$ and $y_1, \cdots, y_m$ in $\RR$, we denote
$\bar x=(x_1, \cdots, x_m)$ and $\bar y=(y_1, \cdots, y_m)$.
The notation $[\bar x,\bar y]$ denotes the rectangle $[x_1,y_1]\times\cdots\times[x_m,y_m]$  in $\RR^m$. The operator $\square^m([\bar x,\bar y])$ is defined as $\square^m([\bar x,\bar y]):=\prod_{k=1}^m \left[V_k(y_k)-V_k(x_k)\right]$.
  % \begin{equation*}
  %   \square^m([\bar x,\bar y]):=\prod_{k=1}^m \left[V_k(y_k)-V_k(x_k)\right]\,.
  % \end{equation*}
  When applied to an $m$-multivariate   function,  $\square^m([\bar x,\bar y])$ is the rectangular increment of the function over the rectangle $[\bar x, \bar y]$. In particular, when $f(x)=f(x_1)\cdots f(x_m)$,  then
$\square^m([\bar x,\bar y])f=\prod_{k=1}^m [f(y_k)-f(x_k)]$.
  Moreover, for sufficiently smooth function $f$, the rectangular increment of $f$ can be computed  as follows
  \begin{equation}\label{id.sq-part}
    \square^m([\bar x,\bar y])f=\int_{[x,y]}\frac{\partial ^m}{\partial z_1 \partial z_2\cdots \partial z_m}f(\bar z)\,d\bar z\,.
  \end{equation}

With these notations, we can write $\prod_{k=1}^m \left(L_B([\xi,\eta],y_k)-L_B([\xi,\eta],x_k)\right)$ as follows
  \begin{equation*}
 \prod_{k=1}^m \left(L_B([\xi,\eta],y_k)-L_B([\xi,\eta],x_k)\right)=   \square^m([\bar x,\bar y]) \prod_{j=1}^m L_B([\xi,\eta],u_j)\,.
  \end{equation*}
Notice that the operator $\square$ also commutes with the expectation operator. In particular, when combined with \eqref{jmoment}, we obtain the formula
\begin{multline}
   \EE \prod_{k=1}^m  \left(L_B([\xi,\eta],y_k)-L_B([\xi,\eta],x_k)\right)
\\= \sum_{\sigma\in\mathfrak{S}_m} \int_{D_m}\!\! d\bar s\,
\square^m([\bar x,\bar y]) \prod_{j=1}^m
p(s_j-s_{j-1},u_{\sigma_j}-u_{\sigma_{j-1}})\,. \label{id.L-sq}
%\\
%&=&\sum_{\sigma\in\mathfrak{S}_m} \int_{D_m}\!\! d\bar s\, \int_{[x,y]}\, d\bar z\,
%\frac{\partial ^m}{\partial z_1\cdots \partial z_m} \prod_{j=1}^m p(s_j-s_{j-1},
%z_{\sigma_j}-z_{\sigma_{j-1}}).\label{id.L-sq2}
  \end{multline}
  % Estimate for the moment $\EE[L_B([\xi,\eta],y)-L_B([\xi,\eta],x)]^{2n}$ is well known (see \cite{geman,marcus}). Hence, we only sketch the idea.
First,  let us assume  $x_1=\cdots=x_m=x$ and $y_1=\cdots=y_m=y$.   Denote $\bar x_{\hat m}=(x_{\si_1}, \cdots, x_{\si_{m-1}})$ and
$\bar x_{\hat m, \widehat{m-1}}=(x_{\si_1}, \cdots, x_{\si_{m-2}})$ etc.  From \eref{id.L-sq} it follows
\begin{align*}
 & \quad  \EE \prod_{k=1}^m  \left(L_B([\xi,\eta],y_k)-L_B([\xi,\eta],x_k)\right)\nonumber\\
%&=& \sum_{\sigma\in\mathfrak{S}_m} \int_{D_m}\!\! d\bar s\, \square^m([\bar x,\bar y])\prod_{j=1}^m p(s_j-s_{j-1},u_{\sigma_j}-u_{\sigma_{j-1}})  \\
&=\sum_{\sigma\in\mathfrak{S}_m} \int_{D_m}\!\! d\bar s\, \square^{m-1}([\bar x_{\hat m} ,\bar y_{\hat m} ])\prod_{j=1}^{m -1}
p(s_j-s_{j-1},u_{\sigma_j}-u_{\sigma_{j-1}}) \\
& \qquad\left[p(s_m-s_{m-1}, y-u_{\si_{m-1}})- p(s_m-s_{m-1}, x-u_{\si_{m-1}})\right]
\\
&=\sum_{\sigma\in\mathfrak{S}_m} \int_{D_m}\!\! d\bar s\, \square^{m-2}([\bar x_{\hat m, \widehat{m-1}} ,\bar y_{\hat m,\widehat{m-1}} ])\prod_{j=1}^{m -2}
p(s_j-s_{j-1},u_{\sigma_j}-u_{\sigma_{j-1}}) \\
&\qquad \bigg\{\left[p(s_m-s_{m-1}, y-y)- p(s_m-s_{m-1}, x-y)\right]p(s_{m-1}-s_{m-2},y-u_{\sigma_{m-2}})\\
& \qquad\quad- \left[p(s_m-s_{m-1}, y-x)- p(s_m-s_{m-1}, x-x)\right]p(s_{m-1}-s_{m-2},x-u_{\sigma_{m-2}}) \bigg\}\\
&=\sum_{\sigma\in\mathfrak{S}_m} \int_{D_m}\!\! d\bar s\, \square^{m-2}([\bar x_{\hat m,\widehat{m-1}} ,\bar y_{\hat m,\widehat{m-1}} ])\prod_{j=1}^{m -2}
p(s_j-s_{j-1},u_{\sigma_j}-u_{\sigma_{j-1}}) \\
& \qquad \left[ p(s_{m-1}-s_{m-2},y-u_{\sigma_{m-2}})+p(s_{m-1}-s_{m-2},x-u_{\sigma_{m-2}}) \right] 
\\&\qquad\times \left[p(s_m-s_{m-1}, 0)- p(s_m-s_{m-1}, x-y)\right]\,.
  \end{align*}
If we continue to apply the operator $V$ this way,  we shall obtain
  \begin{multline}
   \EE\left[L([\xi,\eta],x)-L([\xi,\eta],y)\right]^{2n}\\
 =(2n)!\int_{D_{2n}}\prod_{k=2}^{2n}\left[p(s_{k}-s_{k-1},0)+(-1)^{k+1}
 p(s_{k}-s_{k-1},x-y)\right]\left[p(s_{1},x)+p(s_{1},y)\right]d\bar s\,.\label{id.Lxy}
  \end{multline}

  The estimate \eqref{ineq.Lxy} follows from \eqref{id.Lxy} and the following inequality
  \begin{multline*}
     \int_a^b \int_s^b (t-s)^\gamma[p(t-s,0)-p(t-s,x-y)][p(s-a,0)+p(s-a,x-y)]dtds
     \\\le c_{\beta,\gamma}|b-a|^{\gamma+1-\beta}|x-y|^{2 \beta}\,,
  \end{multline*}
 is valid for all $\beta\in [0,1/2]$ and $\gamma\ge 0$.

Now  we assume a condition which is slightly more restricted than
\eref{cond:xyk}:
\begin{equation}
\label{e.xyk}
    x_1<y_1<x_2<y_2<\cdots<x_{2n}<y_{2n}\,.
\end{equation}
% We also denote $[\bar x,\bar y]=\prod_{k=1}^{2n}[x_k,y_k]$
% \begin{equation}
% [x,y]=\left\{(z_1,\cdots, z_{m})\,:\ \
% x_1<z_1<y_1<x_2<z_2<y_2<\cdots<x_{m}<z_{2n}<y_{m}\right\} \,.
% \end{equation}
The functions $p(t, x)$ and   all its partial derivatives are
continuously differentiable on the any interval $(t, x)\in [0,
\infty)\times (-\infty, -a]\cup [a, \infty)$ for any positive $a$.
Thus the function on the right hand side of \eqref{jmoment} is continuously differentiable on
the $[\bar x,\bar y]$ satisfying \eref{e.xyk}. Using the equations
\eref{id.sq-part}, \eref{id.L-sq} and interchanging order of integrations, we have
\begin{multline}
    \EE \prod_{k=1}^m  \left(L_B([\xi,\eta],y_k)-L_B([\xi,\eta],x_k)\right)
%&=& \sum_{\sigma\in\mathfrak{S}_m} \int_{D_m}\!\! d\bar s\,
%\square^m([\bar x,\bar y])
%\prod_{j=1}^m p(s_j-s_{j-1},u_{\sigma_j}-u_{\sigma_{j-1}})\label{id.L-sq} \\
\\=\sum_{\sigma\in\mathfrak{S}_m} \int_{D_m}\!\! d\bar s\,
\int_{[\bar x,\bar y]}\, d\bar z\, \frac{\partial ^m}{\partial z_1\cdots
\partial z_m} \prod_{j=1}^m p(s_j-s_{j-1},
z_{\sigma_j}-z_{\sigma_{j-1}}).\label{id.L-sq2}
  \end{multline}
   Notice that each partial derivative $\partial/\partial z_{\sigma_j}$
   contributes one derivative to either $p(s_j-s_{j-1},z_{\sigma_j}-
   z_{\sigma_{j-1}})$ or $p(s_{j+1}-s_{j},z_{\sigma_{j+1}}-z_{\sigma_j})$.
   We record the results  by a binary index $e_j$, $e_j=1$ represents the
   former case, $e_j=0$ represents the later case. Moreover, if the later
   case happens, it also contributes a factor $-1$. Since $z_m$ only appears
   in the last term $p(s_{m}-s_{m-1},z_{\sigma_m}-z_{\sigma_{m-1}})$, we must
   have the restriction $e_m=1$. Thus, we can write \eref{id.L-sq2} as
 \begin{multline}
 \EE \prod_{k=1}^m \left(L_B([\xi,\eta],y_k)-L_B([\xi,\eta],x_k)\right)
\\=\sum_{\sigma\in\mathfrak{S}_m}\,
 \sum_{ \bar e\in  \mathfrak{E}  }\sgn( \bar e)   \int_{D_m}\!\! d\bar s
 \, \int_{[\bar x,\bar y]}\, d\bar z\,\prod_{j=1}^m  p^{(e_j+1-e_{j-1})}
 ({s_j-s_{j-1}},z_{\sigma_j}-z_{\sigma_{j-1}})\,,\label{id.Lxyk}
  \end{multline}
where    $\mathfrak{E}$ denotes  all  the $m$-tuple
$\bar e=(e_1,\dots,e_m)\in \{0, 1\}^m$  such that $e_m=1$ and $\sgn (\bar e)$ is the sign of $\bar e$,    defined by $\sgn\bar e:=(-1)^{\sum_{j=1}^m (1-e_j)}$ and $e_0=1$ by convention.

  For instance, in the case $m=4$, when $\sigma$ is the identity map in $\mathfrak{S}_4$ , the integrand in \eqref{id.Lxyk} is
  \begin{equation}\label{id.1.7}
    \begin{split}
    &-p_{s_4-s_3}''(z_4-z_3)p_{s_3-s_2} '(z_3-z_2)p_{s_2-s_1}'(z_2-z_1) p_{s_1}(z_1)\\
    &+p_{s_4-s_3}'(z_4-z_3)p_{s_3-s_2} ''( z_3-z_2)p_{s_2-s_1}'( z_2-z_1) p_{s_1}(z_1)
      \\
    &+p_{s_4-s_3}''(z_4-z_3)p_{s_3-s_2} (z_3-z_2)p_{s_2-s_1}''( z_2-z_1) p_{s_1}(z_1)  \\
    &- p_{s_4-s_3}'(z_4-z_3)p_{s_3-s_2} '(z_3-z_2)p_{s_2-s_1}''( z_2-z_1) p_{s_1}(z_1)\\
    &+p_{s_4-s_3}''(z_4-z_3)p_{s_3-s_2} '(z_3-z_2)p_{s_2-s_1} ( z_2-z_1) p_{s_1}'(z_1) \\
    &-p_{s_4-s_3}'(z_4-z_3)p_{s_3-s_2} ''( z_3-z_2)p_{s_2-s_1} ( z_2-z_1) p_{s_1}'(z_1)
      \\
    &-p_{s_4-s_3}''(z_4-z_3)p_{s_3-s_2} (z_3-z_2)p_{s_2-s_1}' ( z_2-z_1) p_{s_1}'(z_1)  \\
    &+p_{s_4-s_3}'(z_4-z_3)p_{s_3-s_2} '(z_3-z_2)p_{s_2-s_1}' ( z_2-z_1) p_{s_1}'(z_1)\,.
    \end{split}
  \end{equation}

  Combining the estimate in Lemma \ref{lem.Dm} (below) with \eqref{id.Lxyk}, we see that there exists a constant $c_n$ depending only on $n$ such that
  \begin{equation}\label{ineq.Lt1}
    \left|\EE \prod_{k=1}^{2n} \left(L_B([\xi,\eta],y_k)-L_B([\xi,\eta],x_k)\right)\right|\le c_{n} \prod_{j=1}^{2n} |x_j-y_j|\,.
  \end{equation}
  An application   \eqref{ineq.Lxy} with $\beta=0$ yields
  \begin{equation}\label{ineq.Lt2}
    \left|\EE \prod_{k=1}^{2n} \left(L_B([\xi,\eta],y_k)-L_B([\xi,\eta],x_k)\right)\right|\le c_{n}|\eta-\xi|^{n}\,.
  \end{equation}
  Now given $\alpha\in[0,1]$, an interpolating between \eqref{ineq.Lt1} and \eqref{ineq.Lt2} yields
  \begin{equation*}
        \left|\EE \prod_{k=1}^{2n} \left(L_B([\xi,\eta],y_k)-L_B([\xi,\eta],x_k)\right)\right|\le c_{\alpha,n}|\eta-\xi|^{n\alpha}\prod_{j=1}^{2n} |x_j-y_j|^{1-\alpha}\,.
  \end{equation*}
This is \eref{ineq.Lxyk} under the condition \eqref{e.xyk}. The
estimate \eqref{ineq.Lxyk} under the general condition
\eqref{cond:xyk} follows by a limiting argument since both sides of
\eqref{ineq.Lxyk} are continuous function of $x_k,y_k$'s. This finishes  the proof of Proposition~\ref{Prop.moment} modulo the proof of the following lemma
which was  used in the above proof.
  \begin{lemma}\label{lem.Dm} Let $\bar e=(e_1,\dots,e_m)$ ($m\ge2$) be an m-tuple
  in $\{0,1\}^m$ such that $e_m=1$ and we take $e_0=1$ by convention.
  Let $u_k$ $(k=1,2,\dots,m)$ be non-zero real numbers and let
  $D_m$ be the domain $\{\bar s \in [\xi,\eta]^m:\xi<s_1<\cdots<s_m<\eta \}$.
  Then the  following estimate holds
  \begin{equation}\label{ineq.Dm}
    \left|\int_{D_m} \prod_{j=1}^m  p^{(e_j+1-e_{j-1})}({s_j-s_{j-1}},u_j)\,d\bar s\right|\le 1\,.
  \end{equation}
  \end{lemma}
  \begin{proof}
    We denote $\cL$ the Laplace transform with respect to the $t$ variable and put
    \begin{equation}
      J=\int_{D_m} \prod_{j=1}^m p^{(e_j-e_{j-1}+1)}(s_j-s_{j-1},u_j)\,d \bar s\,.
    \end{equation}
    Let $*$ denote the convolution operator, i.e. for two functions $f$ and $g$, $f*g(t)=\int_0^t f(s)g(t-s)ds $. Then we can rewrite $J$ into the form
    \begin{equation*}
      J=\int_{\xi}^{\eta} p^{(e_1)} (s_1,z_{\sigma_1}) f(\eta-s_1)  \,d s_1\,,
    \end{equation*}
    where $f$ is the function defined by
    \begin{equation*}
      f(t)=\int_0^t[p^{(e_2-e_{1}+1)}(\cdot,u_2)*\cdots*p^{(e_m-e_{m-1}+1)}(\cdot,u_m)](s)ds\,.
    \end{equation*}
    It is well known (see for example, \cite{gr}, Formula 3.471 (9)  and Formula 8.469 (3))
   that
    \begin{equation}\label{lap.p}
      \cL [p(\cdot,x)](s)=\frac{1}{\sqrt{2s}}e^{-|x|\sqrt{2s}}\,.
    \end{equation}
    By taking derivative under the integral sign (noticing that we assume $x\not=0$), we obtain
    \begin{equation*}
      \cL [ p'(\cdot,x)](s)=-\sgn(x)e^{-|x|\sqrt{2s}}\,.
    \end{equation*}
    We further notice that $p''=2\partial_t p$, thus
    \begin{equation*}
      \cL [p''(\cdot,x)](s)={\sqrt{2s}}e^{-|x|\sqrt{2s}}\,.
    \end{equation*}
    Writing all three formulas in one, for $k=0,1,2$,
    % for $k=0,1,2, \cdots$
    % (we only need the case $k=0,1,2$),
    we have
    \begin{equation}
      \cL [p^{(k)}(\cdot,x)](s)=(\sqrt{2s})^{k-1}[-\sgn(x)]^{k} e^{-|x|\sqrt{2s}}\,.
    \end{equation}
    Since convolution becomes product under Laplace transform, the Laplace transform of $f$ is
    \begin{align*}
      \cL [f](s)&= s^{-1}\prod_{j=2}^m (\sqrt{2s})^{e_j-e_{j-1}}
      [-\sgn(u_j)]^{e_j-e_{j-1}+1}\exp\left\{-|u_j| \sqrt{2s} \right\}\\
      &=\sqrt{2}
      (\sqrt{2s})^{-1-e_1}\exp\left\{-\sqrt{2s}\sum_{j=2}^m|u_j|
      \right\}\prod_{j=2}^m [-\sgn(u_j)]^{e_j-e_{j-1}+1} \,,
    \end{align*}
    where the  factor $s^{-1}$ comes from the fact that the Laplace
    transform of $\int_0^t f(r) dr$ is $s^{-1} \cL f(s)$.
    To simplify notations, we will denote $|u|=\sum_{j=2}^m|u_j|$. We consider now two cases. Case 1: $e_1=0$. Inverting the Laplace transform, using \eqref{lap.p}, we see that
    \begin{equation}
      f(t)=\sqrt{2}\prod_{j=2}^m [-\sgn(u_j)]^{e_j-e_{j-1}+1} p(t,|u|)\,.
    \end{equation}
    Thus
    \begin{align*}
      |J|\le \sqrt2 \int_{\xi}^{\eta}p(s_1,u_1)p(\eta-s_1,|u|)d s_1\le 1/\sqrt2\,.
    \end{align*}
    Case 2: $e_1=1$. We notice that
    \begin{equation*}
      \cL \left[\mathrm{erfc }\left(\frac{|x|}{\sqrt{2t}}\right) \right](s)=\frac1s e^{-|x|\sqrt{2s}}\,,
    \end{equation*}
    where $\mathrm{erfc }(z)$ is the complementary error function
    $\mathrm{erfc }(z) =\frac2{\sqrt\pi}\int_z^\infty e^{-y^2}dy $.
    Inverting the Laplace transform as in the former case, we obtain
    \begin{equation}
      f(t)=\prod_{j=2}^m [-\sgn(u_j)]^{e_j-e_{j-1}+1} \mathrm{erfc }
      \left(\frac{|u|}{\sqrt{2t}} \right)\,.
    \end{equation}
    Thus if we use the fact that $0\le \mathrm{erfc }(z) \le 1$, we
    have
    \begin{align*}
      |J|&\le \int_{\xi}^{\eta}|p'(s_1,u_1)|\mathrm{erfc }
      \left(\frac{|u|}{\sqrt{2(\eta-s_1)}} \right)d s_1\\
      &\le \int_{\xi}^{\eta} \frac{|u_1|}{\sqrt{2\pi s_1}}e^{-\frac{|u_1|^2}{2s_1}}\frac{ds_1}{s_1}\,.
    \end{align*}
    By the change of variable $t=\frac{|u_1|}{\sqrt{2s_1}}$, we see that
     $ |J|\le \frac{2}{\sqrt\pi}\int_{0}^\infty e^{-t^2}dt\le 1$.
  \end{proof}

\section{Proof of Proposition~\ref{prop.convg}}\label{sec.convg}
\setcounter{equation}{0}
    To outline the strategy proving Proposition \ref{prop.convg}, let us first
  observe that using the representation $X_\pi=S_{W_\pi}^{-1}\circ B\circ T_{W_\pi,B}^{-1}$
  we can write
  \begin{multline*}
    \int_0^t g(X_\pi(s),W_\pi\circ (X_\pi(s))\dot{W}_\pi(X_\pi(s))ds
    \\=\int_{-\infty}^\infty g(x,W_\pi(x))e^{-W_\pi(x)} L_B(T^{-1}_{W_\pi,B}(t) ,S_{W_\pi}(x))\dot{W}_\pi(x)\,dx\,.
  \end{multline*}
  We observe that from Lemma~\ref{lemma_inverse}, with probability one, $T_{W_\pi,B}^{-1}(\cdot)$    converges to $T_{W,B}^{-1}(\cdot) $ uniformly over compacts of $\RR_+$. In addition, the function $e^{-u}$ can be combined with $g(x,u)$. Therefore, to prove Proposition \ref{prop.convg}, it suffices to show
  \begin{itemize}
    \item For every function $g$ satisfying conditions \eqref{cond.g1} and \eqref{cond.g2}, with probability one, the process $\xi\mapsto\int_{-\infty}^\infty g(x,W_\pi(x)) L_B(\xi ,S_{W_\pi}(x))\dot{W}_\pi(x)\,dx$ converges
    to $\int_{-\infty}^\infty g(x,W(x))L_B(\xi,S_W(x))W(d^ox)$ uniformly over compact sets.
  \end{itemize}
  The remaining of this section is devoted to verify the previous statement. In what follows, $\{\ell_\pi(g,\xi),\xi\ge0\} $ denote the process
  \begin{equation}\label{def.ellpi}
    \ell_\pi(g,\xi)=\int_\RR g(x,W_\pi(x))L_B(\xi ,S_{W_\pi}(x))\dot{W}_\pi(x)\,dx\,,
  \end{equation}
  which is well-defined for all continuous sample paths of $W$. For every compact set $K$, we denote
  \begin{equation*}
    c_3(K)=c_1(K)+c_2(K)
  \end{equation*}
  where $c_1$ and $c_2$ are the constant in \eqref{cond.g1} and \eqref{cond.g2}.

  In subsection \ref{ss.bded}, we will truncate the processes
  $\ell_\pi$ and show the corresponding truncated processes converges uniformly.
  In subsection \ref{ss.rr}, the claim is verified completely via a gluing argument.

  Let us remark that for all the results in this section holds, we employ the two estimates \eqref{ineq.Lxy} and \eqref{ineq.Lxyk} for the local time of Brownian motion $B$, $L_B$.
  \subsection{Convergence over bounded interval}\label{ss.bded}
  We consider an interval $[a,b]$ with length $L=b-a$. Let $\pi=\{a=x_{0}<x_{1}<\cdots<x_{n}=b\}$
  be a partition of $[a,b]$ with mesh size
  \[
  \Delta=\max_{k=0,\dots,n-1}|x_{k+1}-x_{k}|.
  \]
  We denote
  \begin{equation}\label{def.lab}
  \ell_{\pi}^{[a,b]}(g,\xi)  =  \int_{a}^{b}g(x,W_{\pi}(x))L_{B}(\xi,S_{W_{\pi}}(x))\dot{W}_{\pi}(x)dx\,,
  \end{equation}
where as usual $W_{\pi}$ is the linear interpolation of $W$ associated with $\pi$.

  We first decompose $\ell_{\pi}^{[a,b]}(g,\xi)$ as follows
  \begin{align*}
    \ell_{\pi}^{[a,b]}(g,\xi)
    &=\sum_{k=0}^{n-1} \int_{x_k}^{x_{k+1}}g(x,W_{\pi}(x))[L_B(\xi,S_{W_{\pi}}(x))-L_B(\xi,S_{W_{\pi}}(x_k))]\dot{W}_{\pi}(x)dx
    \\&\quad +\sum_{k=0}^{n-1}L_B(\xi,S_{W_{\pi}}(x_k)) \int_{x_k}^{x_{k+1}}[g(x,W_{\pi}(x))-g(x_k,W_{\pi}(x))]\dot{W}_{\pi}(x)dx
    \\&\quad +\sum_{k=0}^{n-1}L_B(\xi,S_{W_{\pi}}(x_k)) \int_{x_k}^{x_{k+1}}g(x_k,W_{\pi}(x))\dot{W}_{\pi}(x)dx
  \end{align*}
  Let $G$ be a function such that $\partial_u G(x,u)=g(x,u)$.
  The integral inside the last summand can be computed as follows
  \begin{align*}
    \int_{x_k}^{x_{k+1}}g(x_k,W_\pi(x))\dot{W}_\pi(x)dx&=G(x_k,W(x_{k+1}))-G(x_k,W(x_k))
    \\&=\int_{x_k}^{x_{k+1}}g(x_k,W(x))W(dx)+\frac12\int_{x_k}^{x_{k+1}}\partial_ug(x_k,W(x))dx\,,
  \end{align*}
  where the last line follows from the classical It\^o formula. Therefore, we can further decompose $\ell^{[a,b]}_\pi(g,\xi)$ as
  \begin{equation*}
    \ell_{\pi}^{[a,b]}(g,\xi)=I_1(\xi)+I_2(\xi)+I_3(\xi)+I_4(\xi)
  \end{equation*}
  where
  \begin{align*}
    I_1(\xi)&=\sum_{k=0}^{n-1} \int_{x_k}^{x_{k+1}}g(x,W_{\pi}(x))[L_B(\xi,S_{W_{\pi}}(x))-L_B(\xi,S_{W_{\pi}}(x_k))]\dot{W}_{\pi} (x)dx\,,\\
    I_2(\xi)&=\sum_{k=0}^{n-1} L_B(\xi,S_{W_{\pi}}(x_k))\int_{x_k}^{x_{k+1}}[g(x,W_{\pi}(x))-g(x_k,W_{\pi}(x))]\dot{W}_{\pi}(x)dx\,,\\
    I_3(\xi)&=\sum_{k=0}^{n-1} \int_{x_k}^{x_{k+1}}L_B(\xi,S_{W_{\pi}}(x_k))g(x_k,W(x)) W(dx)\,,\\
    I_4(\xi)&=\frac12 \sum_{k=0}^{n-1}\int_{x_k}^{x_{k+1}} \partial_ug(x_k,W(x))L_B(\xi,S_{W_{\pi}}(x_k))dx\,.
  \end{align*}
To simplify notation,  we omit dependence of $I_i$'s on $g$.
  For a generic function $f$ on $\RR$, we will denote $$f([\xi,\eta])\equiv f(\eta)-f(\xi),\;\forall \eta,\xi\in \RR.$$

  \begin{lemma}\label{lem.bded} Suppose $g$ satisfies the conditions in Proposition~\ref{prop.convg}. There exist positive constants $\epsilon, \gamma, \kappa$ which does not depend on $(a,b)$ such that the following estimates holds: for all $\eta,\xi\in\RR_+$,
  \begin{equation}
  	\EE |I_1([\xi,\eta])|^6\lesssim c_1^6(b-a) e^{\kappa (|a|\vee|b|)} |\eta- \xi|^{1+\epsilon} \Delta^{\gamma}\,,\label{est.I1}
  \end{equation}
  \begin{equation}
  	\EE |I_2([\xi,\eta])|^6 \lesssim c_2^6(b-a) e^{\kappa (|a|\vee|b|)}|\eta- \xi|^{1+\epsilon} \Delta^{\gamma}\,,\label{est.I2}
  \end{equation}
  \begin{multline}
  	\EE |I_3([\xi,\eta])-\int_{a}^{b}g(x,W(x))L_B([\xi,\eta],S_W(x))W(dx)|^6 
  	\\\lesssim c_3^6(b-a)  e^{\kappa (|a|\vee|b|)}|\eta- \xi|^{1+\epsilon} \Delta^{\gamma}\,,\label{est.I3}
  \end{multline}
  \begin{multline}
  	\EE |I_4([\xi,\eta])- \frac12\int_{a}^{b}\partial_ug(x,W(x))L_B([\xi,\eta],S_W(x))dx|^6 
  	\\\lesssim c_3^6(b-a) e^{\kappa (|a|\vee|b|)}|\eta- \xi|^{1+\epsilon} \Delta^{\gamma}\,,\label{est.I4}
  \end{multline}
    where the implied constants depend only on $b-a$. As a consequence, for all $\eta,\xi\in\RR_+\,,$
    \begin{multline}\label{est.lab}
      \EE |\ell_\pi^{[a,b]}(g,[\xi,\eta]) -\int_a^b g(x,W(x))L_B([\xi,\eta],S_W(x))W(d^ox) |^6
      \\\lesssim c_3^6(b-a) e^{\kappa (|a|\vee|b|)}|\eta- \xi|^{1+\epsilon} \Delta^{\gamma}  \,.
    \end{multline}
  \end{lemma}
  \begin{proof}
    To deal with $I_1$,  we denote
    \[a_k=\int_{x_{k}}^{x_{k+1}}g(x,W_{\pi}(x))\left[L_{B}([\xi,\eta],S_{W_{\pi}}(x))-L_{B}([\xi,\eta],S_{W_{\pi}}(x_{k}))\right]\dot{W}_{\pi}(x) dx\,.\]
    Then
    \begin{align*}
      &\EE I_1([\xi,\eta])^6 
      \\&= \sum_{k_1=0}^{n-1}\EE a_{k_1}^6+  6\sum_{k_1\neq k_2}\EE a^5_{ k_1}a_{k_2}
       + 15\sum_{k_1,k_2}\EE a^4_{k_1}a^2_{k_2}+30\sum_{k_1,k_2,k_3}\EE a^4_{k_1}a_{k_2}a_{k_3}\\
      &\quad+ 20\sum_{k_1,k_2}\EE a^3_{k_1}a^3_{k_2}+60\sum_{k_1,k_2,k_3}\EE a^3_{k_1}a^2_{k_2}a_{k_3}+120\sum_{k_1,k_2,k_3,k_4}\EE a^3_{k_1}a_{k_2}a_{k_3}a_{k_4} \\
      &\quad+90\sum_{k_1,k_2,k_3}\EE a^2_{k_1}a^2_{k_2}a^2_{k_3}+180\sum_{k_1,k_2,k_3,k_4}\EE a^2_{k_1}a^2_{k_2}a_{k_3}a_{k_4}
      \\&\quad+360\sum_{k_1,k_2,k_3,k_4,k_5}\EE a^2_{k_1}a_{k_2}a_{k_3}a_{k_4}a_{k_5} 
      +6!\sum_{k_1,k_2,k_3,k_4,k_5,k_6}\EE a_{k_1}a_{k_2}a_{k_3}a_{k_4}a_{k_5}a_{k_6}
    \end{align*}
    where the indices $k_1,\dots,k_6$ are pairwise disjoint if they appear under
    the same summation notation. Among these sums, the most difficult
    term to estimate is the last one. All other sums can be handled
    by mean of the H\"older inequality and \eqref{ineq.Lxy}
    (similar to the method of estimating $A$ below).
    To illustrate our method while maintain a decent length of the paper,
    we will give detailed estimates for the two sums
    \begin{eqnarray*}
      &&A = \sum_{k_1,k_2,k_3,k_4,k_5}\EE a^2_{k_1}a_{k_2}a_{k_3}a_{k_4}a_{k_5}, \\
      &&\tilde A =  \sum_{k_1,k_2,k_3,k_4,k_5,k_6}\EE a_{k_1}a_{k_2}a_{k_3}a_{k_4}a_{k_5}a_{k_6}\,.
    \end{eqnarray*}
    To avoid lengthy formula, we denote $\Delta_k W=W(x_{k+1})-W(x_{k})$,
    $\Delta_k=x_{k+1}-x_{k}$. We also omit the indices under the sigma notation. By
    the Cauchy-Schwarz inequality and \eqref{cond.g1}
    \begin{equation*}
      a_k^2\le c_1^2(b-a)\frac{|\Delta_kW|^2}{\Delta_k}\int_{x_{k}}^{x_{k+1}}e^{2\theta |W_{\pi}(z)|}[L_B([\xi,\eta], S_{W_{\pi}}(z))-L_B([\xi,\eta],S_{W_{\pi}}(x_{k}))]^2dz\,.
    \end{equation*}
    Hence, $A$ is bounded  from the above by
    \begin{multline*}
      c_1^6(b-a)\sum \EE\int_{x_{k_1}}^{x_{k_1+1}} e^{2 \theta |W_{\pi}(z)|}[L_B([\xi,\eta], S_{W_{\pi}}(z))-L_B([\xi,\eta],S_{W_{\pi}}(x_{k_1}))]^2
      \\\frac{|\Delta_{k_1}W|^2}{\Delta_{k_1}}dz_1 a_{k_2}a_{k_3}a_{k_4}a_{k_5}\,.
    \end{multline*}
Taking the expectation with respect to the Brownian motion $B$ first
and applying \eqref{ineq.Lxy} with $\beta=1/2$
  we  see that $A$ is bounded  from the above by
    \begin{eqnarray*}
   &&   c_1^6(b-a)|\eta-\xi|^{3/2}  \sum \EE\int_{[x_{k}, {x_{k+1}} ]} e^{2 \theta|W_{\pi}(z_1)|+ \theta |W(z_2)|+\cdots+ \theta |W(z_5)| }
     \\ &&\quad 
 |S_{W_{\pi}}(z_1)-S_{W_{\pi}}(x_{k_1})|\prod_{j=2}^5|S_{W_{\pi}}(z_j)-S_{W_{\pi}}(x_{k_j})|^{1/2}
\frac{|\Delta_{k_1}W|^2}{\Delta_{k_1}}\prod_{j=2}^5
     \frac{|\Delta_{k_j}W| }{\Delta_{k_j}}d\bar z  \,,
    \end{eqnarray*}
    where $\int_{[x_{k},x_{k+1}]}d\bar z$ denotes $\prod_{j=1}^5\int_{x_{k_j}}^{x_{k_j+1}} dz_j$. We further apply the H\"older inequality and the simple estimate
    $\EE e^{\theta |W_{\pi}(z)|}\le e^{\theta^2 |z|/2}$. The above    quality  is
    bounded by a constant multiple of
    \begin{eqnarray*}
    &&  c_1^6(b-a)|\eta-\xi|^{3/2} \sum \int_{[x_{k},x_{k+1}]}
     \left\{\EE |S_{W_{\pi}}(z_1)-S_{W_{\pi}}(x_{k_1})
      |^2\prod_{j=2}^5|S_{W_{\pi}}(z_j)-S_{W_{\pi}}(x_{k_j})| \right\}^{1/2}\\
      &&\qquad  \times e^{\kappa(|z_1|+\cdots+ |z_5|)}  |\Delta_{k_j}|^{-1/2}d\bar z\,.
    \end{eqnarray*}
    Applying the  H\"older inequality again, we obtain
    \begin{eqnarray*}
      A\lesssim c_1^6(b-a)|\eta-\xi|^{3/2}\Delta\left(\int_a^b e^{\kappa |x|}dx\right)^5\,.
    \end{eqnarray*}
    To estimate $\tilde A$, we first take the expectation with respect to the Brownian motion $B$.
    Using \eqref{ineq.Lxyk} with $\alpha\in[0,1]$  we have
    \begin{eqnarray*}
\tilde A \lesssim c_1^6(b-a)|\eta- \xi|^{3 \alpha}\sum
\prod_{j=1}^6\int_{x_{k_j}}^{x_{k_j+1}}  \EE
e^{\theta W_{\pi}(z_j)}|S_{W_{\pi}}(z_j)-S_{W_{\pi}}(x_{k_j})|^{1-\alpha}\,\frac{\Delta_{k_j}W}{\Delta_{k_j}}\,dz_j\,.
    \end{eqnarray*}
    Applying the H\"older inequality  yields
    \begin{eqnarray}
      \tilde A\lesssim c_1^6(b-a)|\eta- \xi|^{3 \alpha} \Delta^{6(\frac12- \alpha)}\left(\int_a^b e^{\kappa |x|}dx\right)^6\,.
    \end{eqnarray}
    Choosing $\alpha$ between $1/3$ and $1/2$ yields \eqref{est.I1}.

    Proof of \eqref{est.I2}: From the H\"older inequality we have
    \begin{multline*}
      \EE |I_2([\xi,\eta])|^6
      \le (b-a)^5\times
      \\\EE \sum_{k=0}^{n-1}\int_{x_k}^{x_{k+1}} |g(x,W_{\pi}(x))-g(x_k,W_{\pi}(x))|^6 |L_B([\xi,\eta],S_{W_{\pi}}(x_k))|^6|\dot{W}_{\pi}(x)|^6dx\,.
    \end{multline*}
    An further application of the H\"older inequality, condition \eqref{cond.g2} and the estimate \eqref{ineq.Lxy}
    with $\beta=0$ yields
    \begin{equation*}
      \EE |I_1([\xi,\eta])|^6\lesssim (b-a)^5c_2^6(b-a)e^{\kappa(|a|\vee |b|)} |\eta-\xi|^3 \sum_{k=0}^{n-1}|x_{k+1}-x_k|^{6 \lambda-2}\,,
    \end{equation*}
    which implies \eqref{est.I2}. 

    Proof of \eqref{est.I3}: Applying the   moment inequality for martingales, we see that
    the expression on its left hand side is at most a constant times
    \begin{equation*}
      \sum_{k=0}^{n-1}\int_{x_{k}}^{x_{k+1}} \EE\left[g(x,W(x))L_{B}([\xi,\eta],S_{W}(x))-
      g(x_k,W(x)) L_{B}([\xi,\eta],S_{W_{\pi}}(x_{k}))\right]^{6}dx,
    \end{equation*}
    which is again bounded by the sum of a certain constant multiple of
\begin{equation*}
  D:=  c_1^6(b-a)  \sum_{k=0}^{n-1}\int_{x_{k}}^{x_{k+1}} \EE\left[ L_{B}([\xi,\eta],S_{W}(x))- L_{B} ([\xi,\eta],S_{W_{\pi}}(x_{k}))\right]^{6}e^{6 \theta|W(x)|}dx,
\end{equation*}
and
\begin{equation*}
\tilde   D:=      \sum_{k=0}^{n-1}\int_{x_{k}}^{x_{k+1}} \EE\left[
  L_{B}
      ([\xi,\eta],S_{W_{\pi}}(x_{k}))\right]^{6}\left[ g(x,W(x))-g(x_k,W(x))\right]^6 dx,
\end{equation*}
Similar to the estimation for $I_2$, it is easy to see that $\tilde D$ satisfies
\begin{equation*}
    \tilde D\lesssim c_2^6(b-a)e^{\kappa(|a|\vee|b|)}|\eta- \xi|^3\sum_{k=0}^{n-1}|x_{k+1}-x_k|^{6 \lambda-2}
\end{equation*}
which in turn satisfies the bound \eqref{est.I3}.

    By mean of inequality \eqref{ineq.Lxy} with $\beta\in(0,1/2]$, $D$ is bounded
     by a constant times
    \[
    c_1^6(b-a)|\eta- \xi|^{3(1-\beta)}\sum_{k=0}^{n-1}\int_{x_{k}}^{x_{k+1}}\EE|S_{W}(x)-S_{W_{\pi}}(x_{k})|^{6\beta}e^{6 \theta|W(x)|}dx.
    \]
    By the H\"older inequality, we see that above expression is at most a constant times
    \[c_1^6(b-a)|\eta- \xi|^{3(1- \beta)}|\Delta|^{3 \beta}\int_a^b e^{\kappa |x|}dx\,,\]
    which also yields  \eqref{est.I3}.

    Proof of \eqref{est.I4}: By the H\"older inequality, the quality on the left hand side of \eqref{est.I4} is at most a constant times
    \begin{equation*}
    \sum_{k=0}^{n-1}\int_{x_{k}}^{x_{k+1}}  \EE\left[\partial_ug(x,W(x)) L_{B}([\xi,\eta],S_{W}(x))-
    \partial_ug(x_k,W(x)) L_{B}([\xi,\eta],S_{W_{\pi}}(x_{k}))\right]^{6} dx\,.
    \end{equation*}
    From here, \eqref{est.I4} follows similarly.
  \end{proof}

  \subsection{Convergence over $\RR$}\label{ss.rr}
  Let $\gamma$ and $\kappa$ be the constants in Lemma \ref{lem.bded}. Let $\pi$ be a partition of $\RR$. For every $N\in\ZZ$, let $\pi_N$ be the partition on $[N-1,N]$ induced by $\pi$ and $|\pi_N|$ denote the mesh size of $\pi_N$. For every $\delta>0$, we now choose a partition $\pi(\delta)$ such that
  \begin{equation}\label{cond.pidel}
     \sum_{N} c_3([N-1,N]) (e^{\kappa|N|}|\pi_N|^\gamma)^{\frac16} \le\delta \,.
  \end{equation}
  With the notations in the previous subsection, the process $\ell_\pi(g,\cdot)$ (defined in \eqref{def.ellpi}) can be written as
  \begin{equation}
    \ell_\pi(g,\xi)=\sum_{N\in\ZZ} \ell_{\pi_N}^{[N-1,N]}(g,\xi)\,,\; \xi\geq 0,
  \end{equation}
  where $\ell_{\pi_N}^{[N-1,N]}(g,\cdot)$ is the process defined in \eqref{def.lab}. Finiteness of the process $\ell_\pi(g,\cdot)$ will become clear at the end of this subsection. For a random variable $Y$, we denote the $L^6$-norm $\|Y\|_6:=(\EE Y^6)^{1/6}$. To simply notations, we further denote
  \begin{equation*}
    \ell(g,[\xi,\eta])=\int_{-\infty}^\infty g(x,W(x))L_B(\xi,S_W(x))W(d^ox)
  \end{equation*}
  and
  \begin{equation*}
    \ell^{[N-1,N]} (g,[\xi,\eta])=\int_{N-1}^N g(x,W(x))L_B(\xi,S_W(x))W(d^ox)\,.
  \end{equation*} From the estimate \eqref{est.lab}, we obtain
  \begin{align*}
    \|\ell_\pi(g,[\xi,\eta])-\ell(g,[\xi,\eta])\|_6&\le \sum_{N\in\ZZ} \left\|\ell_{g,\pi_N}^{[N-1,N]}(g,[\xi,\eta])-\ell^{[N-1,N]} (g,[\xi,\eta]) \right\|_6\\
    &\lesssim  |\eta- \xi|^{(1+\epsilon)/6} \sum_{N\in\ZZ}c_3([N-1,N]) \left(|\pi_N|^{\gamma} e^{\kappa |N|} \right)^{1/6}\,.
  \end{align*}
  We now choose $\pi=\pi(\delta)$ and use the condition \eqref{cond.pidel} to obtain
  \begin{equation}
    \|\ell_{\pi(\delta)}(g, [\xi,\eta])-\ell(g,[\xi,\eta])\|_6\lesssim  |\eta- \xi|^{(1+\epsilon)/6}\delta\,.
  \end{equation}
  Let $K$ be any positive number. Applying the
  Garsia-Rodemich-Rumsey inequality (see \cite{garsiarodemich}), wee
  see that
  there exists a continuous version of the process $\ell_\pi(g,\cdot)-\ell(g,\cdot)$ which  satisfies the following estimate almost
  surely
  \begin{equation}
    \sup_{0<\xi<\eta<K} \frac{|\ell_{\pi(\delta)}(g,[\xi,\eta])-\ell(g,[\xi,\eta])|}{|\eta-\xi|^{\epsilon/8}}\le C_K  \delta\,.
  \end{equation}
  Since $\ell(g,\cdot)$ has a continuous version and is finite almost surely, this implies the same properties holds for $\ell_{\pi(\delta)}(g,\cdot)$. Moreover, we have also proved the uniform convergence
  \begin{equation}\label{uni.conv}
    \lim_{\delta\to 0}\sup_{0<\xi<\eta<K} \frac{|\ell_{\pi(\delta)}(g,[\xi,\eta])-\ell(g,[\xi,\eta])|}{|\eta-\xi|^{\epsilon/8}}=0.
  \end{equation}
which holds almost surely. This finishes the proof of Step 2, and hence of Proposition~\ref{prop.convg}.


% \bibliography{ref}

\end{document}